\pdfoutput=1 
\documentclass{article}[12pt]

\usepackage{amssymb,amsmath,amsthm,amsfonts,mathrsfs}
\usepackage[hmargin=3cm,vmargin=3.5cm]{geometry}
\usepackage{graphicx} 
\usepackage{hyperref}
\usepackage{amsfonts}
\usepackage{amssymb}
\usepackage{amsmath}
\usepackage{amsthm}
\usepackage{mathtools}
\usepackage{cleveref}
\usepackage{tikz}
\usepackage{tikz-cd}
\usepackage{subcaption}
\usepackage{todonotes} 
\usepackage{enumitem}

\usepackage[
    backend=biber,
    style=alphabetic,
    sorting=nyt,
    doi=false,
    url=false,
    isbn=false,
    maxbibnames=10,
]{biblatex}

\usepackage{fancyvrb}
\VerbatimFootnotes

\theoremstyle{plain}
    \newtheorem{thm}{Theorem}[section]
    \newtheorem{prop}[thm]{Proposition}
    \newtheorem{cor}[thm]{Corollary}
    \newtheorem{lem}[thm]{Lemma}

\theoremstyle{definition}
    
    \newtheorem{ex}[thm]{Example}

\theoremstyle{remark}
	\newtheorem{remark}[thm]{Remark}%
	\newtheorem{question}[thm]{Question}%

\crefname{prop}{Proposition}{Propositions}

\def\R{\mathbb R}
\def\Q{\mathbb Q}
\def\Z{\mathbb Z}

\def\F{\mathbb F}

\newcommand{\mcF}{{\mathcal F}}

\newcommand{\CKh}{\mathit{CKh}}
\newcommand{\Kh}{\mathit{Kh}}
\newcommand{\rCKh}{\widetilde{\CKh}}
\newcommand{\rKh}{\widetilde{\Kh}}

\newcommand{\hsigma}{\widehat{\sigma}}
\newcommand{\hnu}{\widehat{\nu}}
\renewcommand{\sl}{\mathfrak{sl}}

\DeclareMathOperator{\Hom}{Hom}

\DeclareMathOperator{\Tor}{Tor}

\DeclareMathOperator{\Mod}{-Mod}

\DeclareMathOperator{\fchar}{char}

\DeclareMathOperator{\gr}{gr}
\newcommand{\id}{\mathsf{id}}

\DeclareMathOperator{\Cob}{Cob}

\newcommand{\ca}{\alpha}
\newcommand{\cb}{\beta}
\newcommand{\cx}{\gamma^+}
\newcommand{\cy}{\gamma^-}
\newcommand{\cz}{\zeta}

\renewcommand{\epsilon}{\varepsilon}
\renewcommand{\emptyset}{\varnothing}
\newcommand{\bdot}{\bullet}
\newcommand{\hdot}{\circ}

\newcommand{\Xbar}{\overline{X}}
\newcommand{\Ybar}{\overline{Y}}

\newcommand{\dual}{\mathsf{D}}

\title{Symmetries of equivariant Khovanov homology}
\author{Mikhail Khovanov, Taketo Sano}

\begin{document}

\maketitle

\begin{abstract}%
    We study symmetries in equivariant versions of Khovanov homology, which include (i) the construction of an involution $\widehat{\sigma}$ for the $U(2)$-equivariant theory, (ii) an integral lifting $\widehat{\nu}$ of the Shumakovitch operation $\nu$, and (iii) splitting of the $U(1)$- and $U(1)\times U(1)$-equivariant theories generalizing earlier work over $\mathbb{F}_2$. 
Finally, we relate these structures to the Rasmussen $s$-invariant over an arbitrary field $F$.
\end{abstract}

\tableofcontents

\section*{Introduction}
\addcontentsline{toc}{section}{\protect Introduction}%

Deformations and modifications of Khovanov homology~\cite{Khovanov:2000} by E.S.Lee~\cite{Lee:2005} and D.Bar-Natan \cite{BarNatan:2005} can be rethought in the framework of equivariant versions of Khovanov homology. 
The universal theory of that kind is the \textit{$U(2)$-equivariant theory}, originally introduced by Bar-Natan~\cite{BarNatan:2005} via a skein theoretic construction, and then reformulated in the context of \textit{Frobenius extensions} in \cite{Khovanov:2004}. The specific Frobenius extension is given by the ground ring $R = \Z[h, t]$ and the Frobenius algebra $A = R[X]/(X^2 - hX - t)$, from which some of the previously known theories are recovered by the following specializations:
\begin{itemize}
    \item the original construction in \cite{Khovanov:2000} by $(h, t) = (0, 0)$,
    \item Lee's deformation \cite{Lee:2005} by $(h, t) = (0, 1)$ over $R = \Q$,
    \item Bar-Natan's deformation in characteristic 2~\cite{BarNatan:2005} by $(h, t) = (H, 0)$ over $R = \F_2[H]$.
\end{itemize}

Relations among various equivariant theories are summarized in \cite{Khovanov:2022}. 
Furthermore, when considered over the field $\F_2$ of two elements, these theories exhibit additional symmetries:

\begin{enumerate}[label=(\roman*)]
    \item the $U(2)$-equivariant theory over $\F_2$ admits an involution $\sigma$ induced from a Frobenius algebra involution $\sigma \colon X \mapsto X + h$,
    \item $\F_2$-Khovanov homology admits the \textit{Shumakovitch operation} $\nu$ \cite{Shumakovitch:2014}, which is an acyclic differential on the homology group, i.e.\ $\nu^2 = 0$ and the complex with respect to the differential $\nu$ is acyclic, and
    \item $\F_2$-Khovanov homology and $\F_2$-Bar-Natan homology each split into two copies of the respective reduced theory \cite{Shumakovitch:2014,Wigderson:2016}. 
\end{enumerate}

In the first two sections of this paper, we show that the symmetries described above extend to various equivariant theories when the ground ring contains $\Z$. Specifically,

\begin{enumerate}[label=(\roman*)]
    \item the $U(2)$-equivariant Khovanov homology admits an integral lift $\hsigma$ of the involution $\sigma$ (\Cref{sec:involution}),
    \item the $U(2)$-equivariant Khovanov homology admits an integral lift $\hnu$ of the Shumakovitch operation $\nu$, which is again an acyclic differential on the homology group (\Cref{subsec:hnu}), and
    \item the $U(1)$- and $U(1)\times U(1)$- equivariant Khovanov homologies each split into two copies of the respective reduced homology (\Cref{subsec:reduced,subsec:U1xU1-equiv}).
\end{enumerate}

The involution $\hsigma$ and the operation $\hnu$ are not endomorphisms over the ground ring $R = \Z[h, t]$ but rather over its subring $\Z[h^2, t]$. Consequently, the splitting results hold over suitable subrings of the corresponding ground rings.

Finally, in \Cref{sec:rasmussen}, we connect these structures to the Rasmussen $s$-invariant \cite{Rasmussen:2010}, considered over an arbitrary field $F$.

\subsection*{Preliminaries}

We assume that the reader is familiar with the construction of Khovanov homology and its equivariant versions \cite{Khovanov:2000,Lee:2005,BarNatan:2005,Khovanov:2004,Khovanov:2022}. Here, we briefly review the setting of the \textit{$U(2)$-equivariant Khovanov homology}, originally defined in~\cite{BarNatan:2005}.   

Let $R_{h, t}$ denote the graded ring\footnote{
    The grading defined here is opposite from \cite{Khovanov:2000} and the same as \cite{Khovanov:2004}.
} 
$\Z[h, t]$ with $\deg h = 2, \deg t = 4$, and $A_{h, t}$ the graded Frobenius $R_{h, t}$-algebra $R_{h, t}[X]/(X^2 - hX - t)$ with $\deg X = 2$, equipped with the algebra structure (multiplication $m$ and unit $\iota$) inherited from $R_{h, t}[X]$ and the coalgebra structure (comultiplication $\Delta$ and counit $\epsilon$) determined by the counit
\[
    \epsilon(1) = 0,\quad
    \epsilon(X) = 1.
\]
The comultiplication $\Delta$ is given by
\[
    \Delta(1) = 1 \otimes X + X \otimes 1 - h(1 \otimes 1),\quad 
    \Delta(X) = X \otimes X +  t(1 \otimes 1).
\]
For any oriented link diagram $D$, let $\CKh_{h, t}(D)$ denote the Khovanov complex of $D$ obtained from the Frobenius algebra $A_{h, t}$, and $\Kh_{h, t}(D)$ its homology. 

In Bar-Natan's reformulation and generalization of Khovanov homology via \textit{dotted cobordisms} \cite{BarNatan:2005}, the local relations for the $U(2)$-equivariant theory are given by:
\begin{center}
    \input{tikzpictures/loc-relation-dot}
\end{center}
Let $\Cob_{\bdot / l}(B)$ denote the category of dotted cobordisms modulo local relations, defined for each finite 
subset $B \subset \partial D^2$ of boundary points of planar tangles. Here we only consider $B = \emptyset$ and write $\Cob_{\bdot / l} := \Cob_{\bdot / l}(\emptyset)$.
The TQFT $\mcF_{h, t}$ obtained from the Frobenius algebra $A_{h, t}$ is recovered by the tautological functor
\[
    \mcF_{h, t} = \Hom_{\Cob_{\bdot / l}}(\emptyset, -)\colon
    \Cob_{\bdot / l} \to R_{h, t}\Mod,
\]
where $R_{h, t}\Mod$ stands for the category of graded $R_{h,t}$-modules.
In particular, the base ring $R_{h, t}$ is given by evaluations of closed dotted surfaces
\[
    \Hom_{\Cob_{\bdot / l}}(\emptyset, \emptyset) \cong R_{h, t}
\]
where elements $h, t \in R_{h, t}$ correspond to:
\begin{center}
    \tikzset{every picture/.style={line width=0.75pt}} 

\begin{tikzpicture}[x=0.75pt,y=0.75pt,yscale=-1,xscale=1]

\draw  [draw opacity=0] (89.94,40.75) .. controls (89.29,43) and (82.45,44.76) .. (74.11,44.76) .. controls (65.35,44.76) and (58.24,42.81) .. (58.24,40.41) .. controls (58.24,40.26) and (58.27,40.11) .. (58.32,39.97) -- (74.11,40.41) -- cycle ; \draw  [color={rgb, 255:red, 128; green, 128; blue, 128 }  ,draw opacity=1 ] (89.94,40.75) .. controls (89.29,43) and (82.45,44.76) .. (74.11,44.76) .. controls (65.35,44.76) and (58.24,42.81) .. (58.24,40.41) .. controls (58.24,40.26) and (58.27,40.11) .. (58.32,39.97) ;  
\draw  [draw opacity=0][dash pattern={on 0.84pt off 2.51pt}] (58.63,39.33) .. controls (59.9,37.27) and (66.43,35.7) .. (74.28,35.7) .. controls (82.93,35.7) and (89.96,37.6) .. (90.15,39.96) -- (74.28,40.06) -- cycle ; \draw  [color={rgb, 255:red, 128; green, 128; blue, 128 }  ,draw opacity=1 ][dash pattern={on 0.84pt off 2.51pt}] (58.63,39.33) .. controls (59.9,37.27) and (66.43,35.7) .. (74.28,35.7) .. controls (82.93,35.7) and (89.96,37.6) .. (90.15,39.96) ;  
\draw   (58.12,40.41) .. controls (58.12,31.57) and (65.28,24.41) .. (74.11,24.41) .. controls (82.95,24.41) and (90.11,31.57) .. (90.11,40.41) .. controls (90.11,49.24) and (82.95,56.4) .. (74.11,56.4) .. controls (65.28,56.4) and (58.12,49.24) .. (58.12,40.41) -- cycle ;
\draw  [fill={rgb, 255:red, 0; green, 0; blue, 0 }  ,fill opacity=1 ] (67.63,31.68) .. controls (67.63,30.55) and (68.55,29.63) .. (69.68,29.63) .. controls (70.82,29.63) and (71.73,30.55) .. (71.73,31.68) .. controls (71.73,32.81) and (70.82,33.73) .. (69.68,33.73) .. controls (68.55,33.73) and (67.63,32.81) .. (67.63,31.68) -- cycle ;
\draw  [fill={rgb, 255:red, 0; green, 0; blue, 0 }  ,fill opacity=1 ] (77.23,31.68) .. controls (77.23,30.55) and (78.15,29.63) .. (79.28,29.63) .. controls (80.42,29.63) and (81.33,30.55) .. (81.33,31.68) .. controls (81.33,32.81) and (80.42,33.73) .. (79.28,33.73) .. controls (78.15,33.73) and (77.23,32.81) .. (77.23,31.68) -- cycle ;
\draw  [draw opacity=0] (249.56,23.31) .. controls (249,25.25) and (243.11,26.77) .. (235.92,26.77) .. controls (228.36,26.77) and (222.23,25.09) .. (222.23,23.01) .. controls (222.23,22.88) and (222.26,22.76) .. (222.3,22.64) -- (235.92,23.01) -- cycle ; \draw  [color={rgb, 255:red, 128; green, 128; blue, 128 }  ,draw opacity=1 ] (249.56,23.31) .. controls (249,25.25) and (243.11,26.77) .. (235.92,26.77) .. controls (228.36,26.77) and (222.23,25.09) .. (222.23,23.01) .. controls (222.23,22.88) and (222.26,22.76) .. (222.3,22.64) ;  
\draw  [draw opacity=0][dash pattern={on 0.84pt off 2.51pt}] (222.57,22.08) .. controls (223.67,20.31) and (229.29,18.95) .. (236.06,18.95) .. controls (243.51,18.95) and (249.57,20.59) .. (249.75,22.62) -- (236.06,22.71) -- cycle ; \draw  [color={rgb, 255:red, 128; green, 128; blue, 128 }  ,draw opacity=1 ][dash pattern={on 0.84pt off 2.51pt}] (222.57,22.08) .. controls (223.67,20.31) and (229.29,18.95) .. (236.06,18.95) .. controls (243.51,18.95) and (249.57,20.59) .. (249.75,22.62) ;  
\draw   (222.13,23.01) .. controls (222.13,15.4) and (228.3,9.22) .. (235.92,9.22) .. controls (243.53,9.22) and (249.71,15.4) .. (249.71,23.01) .. controls (249.71,30.63) and (243.53,36.8) .. (235.92,36.8) .. controls (228.3,36.8) and (222.13,30.63) .. (222.13,23.01) -- cycle ;
\draw  [fill={rgb, 255:red, 0; green, 0; blue, 0 }  ,fill opacity=1 ] (230.33,15.49) .. controls (230.33,14.51) and (231.12,13.72) .. (232.1,13.72) .. controls (233.07,13.72) and (233.87,14.51) .. (233.87,15.49) .. controls (233.87,16.47) and (233.07,17.26) .. (232.1,17.26) .. controls (231.12,17.26) and (230.33,16.47) .. (230.33,15.49) -- cycle ;
\draw  [fill={rgb, 255:red, 0; green, 0; blue, 0 }  ,fill opacity=1 ] (238.61,15.49) .. controls (238.61,14.51) and (239.4,13.72) .. (240.37,13.72) .. controls (241.35,13.72) and (242.14,14.51) .. (242.14,15.49) .. controls (242.14,16.47) and (241.35,17.26) .. (240.37,17.26) .. controls (239.4,17.26) and (238.61,16.47) .. (238.61,15.49) -- cycle ;

\draw  [draw opacity=0] (249.56,60.11) .. controls (249,62.05) and (243.11,63.57) .. (235.92,63.57) .. controls (228.36,63.57) and (222.23,61.89) .. (222.23,59.81) .. controls (222.23,59.68) and (222.26,59.56) .. (222.3,59.44) -- (235.92,59.81) -- cycle ; \draw  [color={rgb, 255:red, 128; green, 128; blue, 128 }  ,draw opacity=1 ] (249.56,60.11) .. controls (249,62.05) and (243.11,63.57) .. (235.92,63.57) .. controls (228.36,63.57) and (222.23,61.89) .. (222.23,59.81) .. controls (222.23,59.68) and (222.26,59.56) .. (222.3,59.44) ;  
\draw  [draw opacity=0][dash pattern={on 0.84pt off 2.51pt}] (222.57,58.88) .. controls (223.67,57.11) and (229.29,55.75) .. (236.06,55.75) .. controls (243.51,55.75) and (249.57,57.39) .. (249.75,59.42) -- (236.06,59.51) -- cycle ; \draw  [color={rgb, 255:red, 128; green, 128; blue, 128 }  ,draw opacity=1 ][dash pattern={on 0.84pt off 2.51pt}] (222.57,58.88) .. controls (223.67,57.11) and (229.29,55.75) .. (236.06,55.75) .. controls (243.51,55.75) and (249.57,57.39) .. (249.75,59.42) ;  
\draw   (222.13,59.81) .. controls (222.13,52.2) and (228.3,46.02) .. (235.92,46.02) .. controls (243.53,46.02) and (249.71,52.2) .. (249.71,59.81) .. controls (249.71,67.43) and (243.53,73.6) .. (235.92,73.6) .. controls (228.3,73.6) and (222.13,67.43) .. (222.13,59.81) -- cycle ;
\draw  [fill={rgb, 255:red, 0; green, 0; blue, 0 }  ,fill opacity=1 ] (230.33,52.29) .. controls (230.33,51.31) and (231.12,50.52) .. (232.1,50.52) .. controls (233.07,50.52) and (233.87,51.31) .. (233.87,52.29) .. controls (233.87,53.27) and (233.07,54.06) .. (232.1,54.06) .. controls (231.12,54.06) and (230.33,53.27) .. (230.33,52.29) -- cycle ;
\draw  [fill={rgb, 255:red, 0; green, 0; blue, 0 }  ,fill opacity=1 ] (238.61,52.29) .. controls (238.61,51.31) and (239.4,50.52) .. (240.37,50.52) .. controls (241.35,50.52) and (242.14,51.31) .. (242.14,52.29) .. controls (242.14,53.27) and (241.35,54.06) .. (240.37,54.06) .. controls (239.4,54.06) and (238.61,53.27) .. (238.61,52.29) -- cycle ;

\draw  [draw opacity=0] (195.94,39.76) .. controls (195.29,42) and (188.45,43.77) .. (180.11,43.77) .. controls (171.35,43.77) and (164.24,41.82) .. (164.24,39.41) .. controls (164.24,39.26) and (164.27,39.12) .. (164.32,38.98) -- (180.11,39.41) -- cycle ; \draw  [color={rgb, 255:red, 128; green, 128; blue, 128 }  ,draw opacity=1 ] (195.94,39.76) .. controls (195.29,42) and (188.45,43.77) .. (180.11,43.77) .. controls (171.35,43.77) and (164.24,41.82) .. (164.24,39.41) .. controls (164.24,39.26) and (164.27,39.12) .. (164.32,38.98) ;  
\draw  [draw opacity=0][dash pattern={on 0.84pt off 2.51pt}] (164.63,38.33) .. controls (165.9,36.27) and (172.43,34.71) .. (180.28,34.71) .. controls (188.93,34.71) and (195.96,36.6) .. (196.15,38.96) -- (180.28,39.06) -- cycle ; \draw  [color={rgb, 255:red, 128; green, 128; blue, 128 }  ,draw opacity=1 ][dash pattern={on 0.84pt off 2.51pt}] (164.63,38.33) .. controls (165.9,36.27) and (172.43,34.71) .. (180.28,34.71) .. controls (188.93,34.71) and (195.96,36.6) .. (196.15,38.96) ;  
\draw   (164.12,39.41) .. controls (164.12,30.58) and (171.28,23.42) .. (180.11,23.42) .. controls (188.95,23.42) and (196.11,30.58) .. (196.11,39.41) .. controls (196.11,48.25) and (188.95,55.41) .. (180.11,55.41) .. controls (171.28,55.41) and (164.12,48.25) .. (164.12,39.41) -- cycle ;
\draw  [fill={rgb, 255:red, 0; green, 0; blue, 0 }  ,fill opacity=1 ] (170.63,30.49) .. controls (170.63,29.35) and (171.55,28.44) .. (172.68,28.44) .. controls (173.82,28.44) and (174.73,29.35) .. (174.73,30.49) .. controls (174.73,31.62) and (173.82,32.54) .. (172.68,32.54) .. controls (171.55,32.54) and (170.63,31.62) .. (170.63,30.49) -- cycle ;
\draw  [fill={rgb, 255:red, 0; green, 0; blue, 0 }  ,fill opacity=1 ] (178.23,30.49) .. controls (178.23,29.35) and (179.15,28.44) .. (180.28,28.44) .. controls (181.42,28.44) and (182.33,29.35) .. (182.33,30.49) .. controls (182.33,31.62) and (181.42,32.54) .. (180.28,32.54) .. controls (179.15,32.54) and (178.23,31.62) .. (178.23,30.49) -- cycle ;
\draw  [fill={rgb, 255:red, 0; green, 0; blue, 0 }  ,fill opacity=1 ] (185.83,30.49) .. controls (185.83,29.35) and (186.75,28.44) .. (187.88,28.44) .. controls (189.02,28.44) and (189.93,29.35) .. (189.93,30.49) .. controls (189.93,31.62) and (189.02,32.54) .. (187.88,32.54) .. controls (186.75,32.54) and (185.83,31.62) .. (185.83,30.49) -- cycle ;

\draw (204.5,30.9) node [anchor=north west][inner sep=0.75pt]    {$-$};
\draw (21,31.61) node [anchor=north west][inner sep=0.75pt]    {$h\ =\ $};
\draw (127,31.61) node [anchor=north west][inner sep=0.75pt]    {$t\ =\ $};
\draw (97,31.61) node [anchor=north west][inner sep=0.75pt]    {$,$};

\end{tikzpicture}
\end{center}
The Frobenius algebra $A_{h, t}$ is given by the state space of a circle
\[
    \Hom_{\Cob_{\bdot / l}}(\emptyset, \bigcirc) \cong A_{h, t}
\]
where the two generators correspond to:
\begin{center}
    \tikzset{every picture/.style={line width=0.75pt}} 

\begin{tikzpicture}[x=0.75pt,y=0.75pt,yscale=-1,xscale=1]

\draw  [draw opacity=0][dash pattern={on 0.84pt off 2.51pt}] (25.7,33.61) .. controls (23.49,31.96) and (21.86,26.52) .. (21.86,20.07) .. controls (21.86,13.63) and (23.48,8.2) .. (25.7,6.54) -- (27.13,20.07) -- cycle ; \draw  [dash pattern={on 0.84pt off 2.51pt}] (25.7,33.61) .. controls (23.49,31.96) and (21.86,26.52) .. (21.86,20.07) .. controls (21.86,13.63) and (23.48,8.2) .. (25.7,6.54) ;  
\draw  [draw opacity=0] (26.55,6.1) .. controls (26.74,6.04) and (26.93,6.02) .. (27.13,6.02) .. controls (30.04,6.02) and (32.4,12.31) .. (32.4,20.07) .. controls (32.4,27.84) and (30.04,34.13) .. (27.13,34.13) .. controls (26.93,34.13) and (26.74,34.11) .. (26.55,34.05) -- (27.13,20.07) -- cycle ; \draw   (26.55,6.1) .. controls (26.74,6.04) and (26.93,6.02) .. (27.13,6.02) .. controls (30.04,6.02) and (32.4,12.31) .. (32.4,20.07) .. controls (32.4,27.84) and (30.04,34.13) .. (27.13,34.13) .. controls (26.93,34.13) and (26.74,34.11) .. (26.55,34.05) ;  
\draw  [draw opacity=0] (27.71,6.02) .. controls (27.51,6.02) and (27.32,6.02) .. (27.13,6.02) .. controls (18,6.02) and (10.61,12.31) .. (10.61,20.07) .. controls (10.61,27.84) and (18,34.13) .. (27.13,34.13) .. controls (27.32,34.13) and (27.51,34.13) .. (27.71,34.12) -- (27.13,20.07) -- cycle ; \draw   (27.71,6.02) .. controls (27.51,6.02) and (27.32,6.02) .. (27.13,6.02) .. controls (18,6.02) and (10.61,12.31) .. (10.61,20.07) .. controls (10.61,27.84) and (18,34.13) .. (27.13,34.13) .. controls (27.32,34.13) and (27.51,34.13) .. (27.71,34.12) ;  

\draw  [draw opacity=0][dash pattern={on 0.84pt off 2.51pt}] (117.7,34.61) .. controls (115.49,32.96) and (113.86,27.52) .. (113.86,21.07) .. controls (113.86,14.63) and (115.48,9.2) .. (117.7,7.54) -- (119.13,21.07) -- cycle ; \draw  [dash pattern={on 0.84pt off 2.51pt}] (117.7,34.61) .. controls (115.49,32.96) and (113.86,27.52) .. (113.86,21.07) .. controls (113.86,14.63) and (115.48,9.2) .. (117.7,7.54) ;  
\draw  [draw opacity=0] (118.55,7.1) .. controls (118.74,7.04) and (118.93,7.02) .. (119.13,7.02) .. controls (122.04,7.02) and (124.4,13.31) .. (124.4,21.07) .. controls (124.4,28.84) and (122.04,35.13) .. (119.13,35.13) .. controls (118.93,35.13) and (118.74,35.11) .. (118.55,35.05) -- (119.13,21.07) -- cycle ; \draw   (118.55,7.1) .. controls (118.74,7.04) and (118.93,7.02) .. (119.13,7.02) .. controls (122.04,7.02) and (124.4,13.31) .. (124.4,21.07) .. controls (124.4,28.84) and (122.04,35.13) .. (119.13,35.13) .. controls (118.93,35.13) and (118.74,35.11) .. (118.55,35.05) ;  
\draw  [draw opacity=0] (119.71,7.02) .. controls (119.51,7.02) and (119.32,7.02) .. (119.13,7.02) .. controls (110,7.02) and (102.61,13.31) .. (102.61,21.07) .. controls (102.61,28.84) and (110,35.13) .. (119.13,35.13) .. controls (119.32,35.13) and (119.51,35.13) .. (119.71,35.12) -- (119.13,21.07) -- cycle ; \draw   (119.71,7.02) .. controls (119.51,7.02) and (119.32,7.02) .. (119.13,7.02) .. controls (110,7.02) and (102.61,13.31) .. (102.61,21.07) .. controls (102.61,28.84) and (110,35.13) .. (119.13,35.13) .. controls (119.32,35.13) and (119.51,35.13) .. (119.71,35.12) ;  
\draw  [fill={rgb, 255:red, 0; green, 0; blue, 0 }  ,fill opacity=1 ] (105.49,20.82) .. controls (105.49,19.96) and (106.19,19.27) .. (107.04,19.27) .. controls (107.9,19.27) and (108.6,19.96) .. (108.6,20.82) .. controls (108.6,21.68) and (107.9,22.37) .. (107.04,22.37) .. controls (106.19,22.37) and (105.49,21.68) .. (105.49,20.82) -- cycle ;

\draw (29.13,20.07) node [anchor=west] [inner sep=0.75pt]    {$\ \leftrightarrow \ 1,$};
\draw (121.13,21.07) node [anchor=west] [inner sep=0.75pt]    {$\ \leftrightarrow \ X$};

\end{tikzpicture}
\end{center}
and the operations of $A_{h, t}$ correspond to:
\begin{center}
    \input{tikzpictures/frobenius-ops-cob}
\end{center}
Here, cobordisms are drawn so that the time axis $I$ runs from left to right. From the local relations, two dots on the component can be reduced as:
\begin{center}
    \tikzset{every picture/.style={line width=0.75pt}} 

\begin{tikzpicture}[x=0.75pt,y=0.75pt,yscale=-.75,xscale=.75]

\draw  [dash pattern={on 0.84pt off 2.51pt}] (10.93,11.56) -- (56.56,11.56) -- (56.56,55.49) -- (10.93,55.49) -- cycle ;
\draw  [dash pattern={on 0.84pt off 2.51pt}] (122.41,11.06) -- (169.08,11.06) -- (169.08,55.99) -- (122.41,55.99) -- cycle ;
\draw  [dash pattern={on 0.84pt off 2.51pt}] (224.95,12.06) -- (269.54,12.06) -- (269.54,54.99) -- (224.95,54.99) -- cycle ;
\draw  [fill={rgb, 255:red, 0; green, 0; blue, 0 }  ,fill opacity=1 ] (23.5,34.2) .. controls (23.5,32.43) and (24.93,31) .. (26.7,31) .. controls (28.47,31) and (29.91,32.43) .. (29.91,34.2) .. controls (29.91,35.97) and (28.47,37.41) .. (26.7,37.41) .. controls (24.93,37.41) and (23.5,35.97) .. (23.5,34.2) -- cycle ;
\draw  [fill={rgb, 255:red, 0; green, 0; blue, 0 }  ,fill opacity=1 ] (36,34.2) .. controls (36,32.43) and (37.43,31) .. (39.2,31) .. controls (40.97,31) and (42.41,32.43) .. (42.41,34.2) .. controls (42.41,35.97) and (40.97,37.41) .. (39.2,37.41) .. controls (37.43,37.41) and (36,35.97) .. (36,34.2) -- cycle ;
\draw  [fill={rgb, 255:red, 0; green, 0; blue, 0 }  ,fill opacity=1 ] (142.54,33.52) .. controls (142.54,31.75) and (143.97,30.32) .. (145.74,30.32) .. controls (147.51,30.32) and (148.95,31.75) .. (148.95,33.52) .. controls (148.95,35.29) and (147.51,36.73) .. (145.74,36.73) .. controls (143.97,36.73) and (142.54,35.29) .. (142.54,33.52) -- cycle ;

\draw (68,25.32) node [anchor=north west][inner sep=0.75pt]    {$=$};
\draw (183,25.32) node [anchor=north west][inner sep=0.75pt]    {$+$};
\draw (103,25.32) node [anchor=north west][inner sep=0.75pt]    {$h$};
\draw (207,25.32) node [anchor=north west][inner sep=0.75pt]    {$t$};

\end{tikzpicture}
\end{center}
corresponding to the identity $X^2 = hX + t$ in $A_{h, t}$. We also let $Y = X - h$, and denote the corresponding element by a \textit{hollow dot}%
\footnote{
    Our definition of the hollow dot $\hdot$ differs by an overall sign from the one defined in \cite{BHP:2023} and in \cite{Khovanov:2004}.
} 
\begin{center}
    \tikzset{every picture/.style={line width=0.75pt}} 

\begin{tikzpicture}[x=0.75pt,y=0.75pt,yscale=-.75,xscale=.75]

\draw  [dash pattern={on 0.84pt off 2.51pt}] (10.93,11.56) -- (56.56,11.56) -- (56.56,55.49) -- (10.93,55.49) -- cycle ;
\draw  [dash pattern={on 0.84pt off 2.51pt}] (101.41,12.06) -- (148.08,12.06) -- (148.08,56.99) -- (101.41,56.99) -- cycle ;
\draw  [color={rgb, 255:red, 0; green, 0; blue, 0 }  ,draw opacity=1 ][fill={rgb, 255:red, 255; green, 255; blue, 255 }  ,fill opacity=1 ][line width=0.75]  (29.5,34.2) .. controls (29.5,32.43) and (30.93,31) .. (32.7,31) .. controls (34.47,31) and (35.91,32.43) .. (35.91,34.2) .. controls (35.91,35.97) and (34.47,37.41) .. (32.7,37.41) .. controls (30.93,37.41) and (29.5,35.97) .. (29.5,34.2) -- cycle ;
\draw  [fill={rgb, 255:red, 0; green, 0; blue, 0 }  ,fill opacity=1 ] (121.54,34.52) .. controls (121.54,32.75) and (122.97,31.32) .. (124.74,31.32) .. controls (126.51,31.32) and (127.95,32.75) .. (127.95,34.52) .. controls (127.95,36.29) and (126.51,37.73) .. (124.74,37.73) .. controls (122.97,37.73) and (121.54,36.29) .. (121.54,34.52) -- cycle ;
\draw  [dash pattern={on 0.84pt off 2.51pt}] (209.95,11.91) -- (254.54,11.91) -- (254.54,54.84) -- (209.95,54.84) -- cycle ;

\draw (68,25.32) node [anchor=north west][inner sep=0.75pt]    {$=$};
\draw (167,24.24) node [anchor=north west][inner sep=0.75pt]    {$-$};
\draw (194,24.17) node [anchor=north west][inner sep=0.75pt]    {$h$};

\end{tikzpicture}
\end{center}
Using the hollow dot, (NC) can be rewritten:
\begin{center}
    \input{tikzpictures/hollow-dot-NC}
\end{center}
The following relations are also useful:
\begin{center}
    \tikzset{every picture/.style={line width=0.75pt}} 

\begin{tikzpicture}[x=0.75pt,y=0.75pt,yscale=-.75,xscale=.75]

\draw  [dash pattern={on 0.84pt off 2.51pt}] (10.93,11.56) -- (56.56,11.56) -- (56.56,55.49) -- (10.93,55.49) -- cycle ;
\draw  [dash pattern={on 0.84pt off 2.51pt}] (128.41,11.06) -- (175.08,11.06) -- (175.08,55.99) -- (128.41,55.99) -- cycle ;
\draw  [dash pattern={on 0.84pt off 2.51pt}] (230.95,12.06) -- (275.54,12.06) -- (275.54,54.99) -- (230.95,54.99) -- cycle ;
\draw  [color={rgb, 255:red, 0; green, 0; blue, 0 }  ,draw opacity=1 ][fill={rgb, 255:red, 255; green, 255; blue, 255 }  ,fill opacity=1 ] (23.5,34.2) .. controls (23.5,32.43) and (24.93,31) .. (26.7,31) .. controls (28.47,31) and (29.91,32.43) .. (29.91,34.2) .. controls (29.91,35.97) and (28.47,37.41) .. (26.7,37.41) .. controls (24.93,37.41) and (23.5,35.97) .. (23.5,34.2) -- cycle ;
\draw  [color={rgb, 255:red, 0; green, 0; blue, 0 }  ,draw opacity=1 ][fill={rgb, 255:red, 255; green, 255; blue, 255 }  ,fill opacity=1 ] (36,34.2) .. controls (36,32.43) and (37.43,31) .. (39.2,31) .. controls (40.97,31) and (42.41,32.43) .. (42.41,34.2) .. controls (42.41,35.97) and (40.97,37.41) .. (39.2,37.41) .. controls (37.43,37.41) and (36,35.97) .. (36,34.2) -- cycle ;
\draw  [color={rgb, 255:red, 0; green, 0; blue, 0 }  ,draw opacity=1 ][fill={rgb, 255:red, 255; green, 255; blue, 255 }  ,fill opacity=1 ] (148.54,33.52) .. controls (148.54,31.75) and (149.97,30.32) .. (151.74,30.32) .. controls (153.51,30.32) and (154.95,31.75) .. (154.95,33.52) .. controls (154.95,35.29) and (153.51,36.73) .. (151.74,36.73) .. controls (149.97,36.73) and (148.54,35.29) .. (148.54,33.52) -- cycle ;
\draw  [dash pattern={on 0.84pt off 2.51pt}] (9.93,68.56) -- (55.56,68.56) -- (55.56,112.49) -- (9.93,112.49) -- cycle ;
\draw  [dash pattern={on 0.84pt off 2.51pt}] (129.95,70.06) -- (174.54,70.06) -- (174.54,112.99) -- (129.95,112.99) -- cycle ;
\draw  [color={rgb, 255:red, 0; green, 0; blue, 0 }  ,draw opacity=1 ][fill={rgb, 255:red, 0; green, 0; blue, 0 }  ,fill opacity=1 ] (22.5,91.2) .. controls (22.5,89.43) and (23.93,88) .. (25.7,88) .. controls (27.47,88) and (28.91,89.43) .. (28.91,91.2) .. controls (28.91,92.97) and (27.47,94.41) .. (25.7,94.41) .. controls (23.93,94.41) and (22.5,92.97) .. (22.5,91.2) -- cycle ;
\draw  [color={rgb, 255:red, 0; green, 0; blue, 0 }  ,draw opacity=1 ][fill={rgb, 255:red, 255; green, 255; blue, 255 }  ,fill opacity=1 ] (35,91.2) .. controls (35,89.43) and (36.43,88) .. (38.2,88) .. controls (39.97,88) and (41.41,89.43) .. (41.41,91.2) .. controls (41.41,92.97) and (39.97,94.41) .. (38.2,94.41) .. controls (36.43,94.41) and (35,92.97) .. (35,91.2) -- cycle ;
\draw  [draw opacity=0] (53.81,147.4) .. controls (53.07,150) and (45.16,152.04) .. (35.51,152.04) .. controls (25.37,152.04) and (17.15,149.78) .. (17.15,147) .. controls (17.15,146.83) and (17.18,146.66) .. (17.24,146.5) -- (35.51,147) -- cycle ; \draw  [color={rgb, 255:red, 128; green, 128; blue, 128 }  ,draw opacity=1 ] (53.81,147.4) .. controls (53.07,150) and (45.16,152.04) .. (35.51,152.04) .. controls (25.37,152.04) and (17.15,149.78) .. (17.15,147) .. controls (17.15,146.83) and (17.18,146.66) .. (17.24,146.5) ;  
\draw  [draw opacity=0][dash pattern={on 0.84pt off 2.51pt}] (17.6,145.75) .. controls (19.07,143.37) and (26.62,141.56) .. (35.71,141.56) .. controls (45.7,141.56) and (53.83,143.75) .. (54.06,146.48) -- (35.71,146.6) -- cycle ; \draw  [color={rgb, 255:red, 128; green, 128; blue, 128 }  ,draw opacity=1 ][dash pattern={on 0.84pt off 2.51pt}] (17.6,145.75) .. controls (19.07,143.37) and (26.62,141.56) .. (35.71,141.56) .. controls (45.7,141.56) and (53.83,143.75) .. (54.06,146.48) ;  
\draw   (17.01,147) .. controls (17.01,136.78) and (25.29,128.5) .. (35.51,128.5) .. controls (45.73,128.5) and (54.01,136.78) .. (54.01,147) .. controls (54.01,157.22) and (45.73,165.5) .. (35.51,165.5) .. controls (25.29,165.5) and (17.01,157.22) .. (17.01,147) -- cycle ;
\draw  [color={rgb, 255:red, 0; green, 0; blue, 0 }  ,draw opacity=1 ][fill={rgb, 255:red, 255; green, 255; blue, 255 }  ,fill opacity=1 ] (32.31,135.7) .. controls (32.31,133.93) and (33.74,132.5) .. (35.51,132.5) .. controls (37.28,132.5) and (38.71,133.93) .. (38.71,135.7) .. controls (38.71,137.47) and (37.28,138.91) .. (35.51,138.91) .. controls (33.74,138.91) and (32.31,137.47) .. (32.31,135.7) -- cycle ;

\draw (67,25.32) node [anchor=north west][inner sep=0.75pt]    {$=$};
\draw (189,25.32) node [anchor=north west][inner sep=0.75pt]    {$+$};
\draw (93,25.32) node [anchor=north west][inner sep=0.75pt]    {$\ -h$};
\draw (213,25.32) node [anchor=north west][inner sep=0.75pt]    {$t$};
\draw (67,82.32) node [anchor=north west][inner sep=0.75pt]    {$=$};
\draw (111,83.32) node [anchor=north west][inner sep=0.75pt]    {$t$};
\draw (67,137.4) node [anchor=north west][inner sep=0.75pt]    {$=\ 1$};

\end{tikzpicture}
\end{center}
We also let $U := X + Y = 2X - h$, and denote the corresponding element by a \textit{star} as in \cite[(16)]{Khovanov:2022}:
\begin{center}
    \tikzset{every picture/.style={line width=0.75pt}} 

\begin{tikzpicture}[x=0.75pt,y=0.75pt,yscale=-.75,xscale=.75]

\draw  [dash pattern={on 0.84pt off 2.51pt}] (10.93,11.56) -- (56.56,11.56) -- (56.56,55.49) -- (10.93,55.49) -- cycle ;
\draw  [dash pattern={on 0.84pt off 2.51pt}] (96.41,12.06) -- (143.08,12.06) -- (143.08,56.99) -- (96.41,56.99) -- cycle ;
\draw  [dash pattern={on 0.84pt off 2.51pt}] (181.95,13.06) -- (226.54,13.06) -- (226.54,55.99) -- (181.95,55.99) -- cycle ;
\draw  [fill={rgb, 255:red, 0; green, 0; blue, 0 }  ,fill opacity=1 ] (116.54,34.52) .. controls (116.54,32.75) and (117.97,31.32) .. (119.74,31.32) .. controls (121.51,31.32) and (122.95,32.75) .. (122.95,34.52) .. controls (122.95,36.29) and (121.51,37.73) .. (119.74,37.73) .. controls (117.97,37.73) and (116.54,36.29) .. (116.54,34.52) -- cycle ;
\draw  [color={rgb, 255:red, 0; green, 0; blue, 0 }  ,draw opacity=1 ][fill={rgb, 255:red, 255; green, 255; blue, 255 }  ,fill opacity=1 ] (201.04,34.52) .. controls (201.04,32.75) and (202.47,31.32) .. (204.24,31.32) .. controls (206.01,31.32) and (207.45,32.75) .. (207.45,34.52) .. controls (207.45,36.29) and (206.01,37.73) .. (204.24,37.73) .. controls (202.47,37.73) and (201.04,36.29) .. (201.04,34.52) -- cycle ;

\draw (68,25.32) node [anchor=north west][inner sep=0.75pt]    {$=$};
\draw (157,26.32) node [anchor=north west][inner sep=0.75pt]    {$+$};
\draw (26,26) node [anchor=north west][inner sep=0.75pt]    {$\star $};

\end{tikzpicture}
\end{center}
From (NC), one can see that a star corresponds to attaching a handle to the surface. Also note that $U^2 = h^2 + 4t$ is the discriminant of the quadratic polynomial $X^2 - hX - t$. 

\section{Involutions}
\label{sec:involution}

\subsection{Involution \texorpdfstring{$\sigma$}{sigma}}

Consider the $R_{h, t}$-algebra involution 
\[
    \sigma\colon R_{h, t}[X] \to R_{h, t}[X],\quad
    X \mapsto h - X
\]
which induces an $R_{h, t}$-algebra involution 
\[
    \sigma\colon A_{h, t} \to A_{h, t}.
\]
Note that $\sigma$ is not a Frobenius algebra isomorphism, since it adds the minus sign to $\epsilon$. Namely, we have

\begin{prop}
\label{prop:sigma-ops}
    \begin{gather*}
        m \circ (\sigma \otimes \sigma) = \sigma \circ m, \quad
        \iota = \sigma \circ \iota, \\
        \Delta \circ \sigma = -(\sigma \otimes \sigma) \circ \Delta,\quad 
        \epsilon \circ \sigma = -\epsilon.
    \end{gather*}
\end{prop}

With the diagrammatic description, the isomorphism $\sigma$ can be expressed as a cobordism:
\begin{center}
    \input{tikzpictures/sigma-cob1}
\end{center}
or as a dotted cobordism:
\begin{center}
    \tikzset{every picture/.style={line width=0.75pt}} 

\begin{tikzpicture}[x=0.75pt,y=0.75pt,yscale=-1,xscale=1]

\draw  [color={rgb, 255:red, 0; green, 0; blue, 0 }  ,draw opacity=1 ] (60.6,28.61) .. controls (60.6,20.9) and (63.18,14.66) .. (66.35,14.66) .. controls (69.53,14.66) and (72.11,20.9) .. (72.11,28.61) .. controls (72.11,36.31) and (69.53,42.56) .. (66.35,42.56) .. controls (63.18,42.56) and (60.6,36.31) .. (60.6,28.61) -- cycle ;
\draw  [draw opacity=0] (65.77,14.56) .. controls (65.97,14.55) and (66.16,14.55) .. (66.35,14.55) .. controls (75.48,14.55) and (82.88,20.84) .. (82.88,28.61) .. controls (82.88,36.37) and (75.48,42.67) .. (66.35,42.67) .. controls (66.16,42.67) and (65.97,42.66) .. (65.77,42.66) -- (66.35,28.61) -- cycle ; \draw   (65.77,14.56) .. controls (65.97,14.55) and (66.16,14.55) .. (66.35,14.55) .. controls (75.48,14.55) and (82.88,20.84) .. (82.88,28.61) .. controls (82.88,36.37) and (75.48,42.67) .. (66.35,42.67) .. controls (66.16,42.67) and (65.97,42.66) .. (65.77,42.66) ;  
\draw  [draw opacity=0][dash pattern={on 0.84pt off 2.51pt}] (118.17,41.61) .. controls (115.95,39.96) and (114.32,34.52) .. (114.32,28.07) .. controls (114.32,21.63) and (115.95,16.2) .. (118.16,14.54) -- (119.59,28.07) -- cycle ; \draw  [dash pattern={on 0.84pt off 2.51pt}] (118.17,41.61) .. controls (115.95,39.96) and (114.32,34.52) .. (114.32,28.07) .. controls (114.32,21.63) and (115.95,16.2) .. (118.16,14.54) ;  
\draw  [draw opacity=0] (119.02,14.1) .. controls (119.21,14.04) and (119.4,14.02) .. (119.59,14.02) .. controls (122.5,14.02) and (124.86,20.31) .. (124.86,28.07) .. controls (124.86,35.84) and (122.5,42.13) .. (119.59,42.13) .. controls (119.4,42.13) and (119.21,42.11) .. (119.02,42.05) -- (119.59,28.07) -- cycle ; \draw   (119.02,14.1) .. controls (119.21,14.04) and (119.4,14.02) .. (119.59,14.02) .. controls (122.5,14.02) and (124.86,20.31) .. (124.86,28.07) .. controls (124.86,35.84) and (122.5,42.13) .. (119.59,42.13) .. controls (119.4,42.13) and (119.21,42.11) .. (119.02,42.05) ;  
\draw  [draw opacity=0] (120.17,14.02) .. controls (119.98,14.02) and (119.79,14.02) .. (119.59,14.02) .. controls (110.47,14.02) and (103.07,20.31) .. (103.07,28.07) .. controls (103.07,35.84) and (110.47,42.13) .. (119.59,42.13) .. controls (119.79,42.13) and (119.98,42.13) .. (120.17,42.12) -- (119.59,28.07) -- cycle ; \draw   (120.17,14.02) .. controls (119.98,14.02) and (119.79,14.02) .. (119.59,14.02) .. controls (110.47,14.02) and (103.07,20.31) .. (103.07,28.07) .. controls (103.07,35.84) and (110.47,42.13) .. (119.59,42.13) .. controls (119.79,42.13) and (119.98,42.13) .. (120.17,42.12) ;  

\draw  [color={rgb, 255:red, 0; green, 0; blue, 0 }  ,draw opacity=1 ] (159.13,28.61) .. controls (159.13,20.9) and (161.71,14.66) .. (164.89,14.66) .. controls (168.07,14.66) and (170.64,20.9) .. (170.64,28.61) .. controls (170.64,36.31) and (168.07,42.56) .. (164.89,42.56) .. controls (161.71,42.56) and (159.13,36.31) .. (159.13,28.61) -- cycle ;
\draw  [draw opacity=0] (164.31,14.56) .. controls (164.5,14.55) and (164.69,14.55) .. (164.89,14.55) .. controls (174.01,14.55) and (181.41,20.84) .. (181.41,28.61) .. controls (181.41,36.37) and (174.01,42.67) .. (164.89,42.67) .. controls (164.69,42.67) and (164.5,42.66) .. (164.31,42.66) -- (164.89,28.61) -- cycle ; \draw   (164.31,14.56) .. controls (164.5,14.55) and (164.69,14.55) .. (164.89,14.55) .. controls (174.01,14.55) and (181.41,20.84) .. (181.41,28.61) .. controls (181.41,36.37) and (174.01,42.67) .. (164.89,42.67) .. controls (164.69,42.67) and (164.5,42.66) .. (164.31,42.66) ;  
\draw  [draw opacity=0][dash pattern={on 0.84pt off 2.51pt}] (216.7,41.61) .. controls (214.49,39.96) and (212.86,34.52) .. (212.86,28.07) .. controls (212.86,21.63) and (214.48,16.2) .. (216.7,14.54) -- (218.13,28.07) -- cycle ; \draw  [dash pattern={on 0.84pt off 2.51pt}] (216.7,41.61) .. controls (214.49,39.96) and (212.86,34.52) .. (212.86,28.07) .. controls (212.86,21.63) and (214.48,16.2) .. (216.7,14.54) ;  
\draw  [draw opacity=0] (217.55,14.1) .. controls (217.74,14.04) and (217.93,14.02) .. (218.13,14.02) .. controls (221.04,14.02) and (223.4,20.31) .. (223.4,28.07) .. controls (223.4,35.84) and (221.04,42.13) .. (218.13,42.13) .. controls (217.93,42.13) and (217.74,42.11) .. (217.55,42.05) -- (218.13,28.07) -- cycle ; \draw   (217.55,14.1) .. controls (217.74,14.04) and (217.93,14.02) .. (218.13,14.02) .. controls (221.04,14.02) and (223.4,20.31) .. (223.4,28.07) .. controls (223.4,35.84) and (221.04,42.13) .. (218.13,42.13) .. controls (217.93,42.13) and (217.74,42.11) .. (217.55,42.05) ;  
\draw  [draw opacity=0] (218.71,14.02) .. controls (218.51,14.02) and (218.32,14.02) .. (218.13,14.02) .. controls (209,14.02) and (201.61,20.31) .. (201.61,28.07) .. controls (201.61,35.84) and (209,42.13) .. (218.13,42.13) .. controls (218.32,42.13) and (218.51,42.13) .. (218.71,42.12) -- (218.13,28.07) -- cycle ; \draw   (218.71,14.02) .. controls (218.51,14.02) and (218.32,14.02) .. (218.13,14.02) .. controls (209,14.02) and (201.61,20.31) .. (201.61,28.07) .. controls (201.61,35.84) and (209,42.13) .. (218.13,42.13) .. controls (218.32,42.13) and (218.51,42.13) .. (218.71,42.12) ;  

\draw  [fill={rgb, 255:red, 0; green, 0; blue, 0 }  ,fill opacity=1 ] (76.32,28.56) .. controls (76.32,27.7) and (77.01,27) .. (77.87,27) .. controls (78.72,27) and (79.42,27.7) .. (79.42,28.56) .. controls (79.42,29.41) and (78.72,30.11) .. (77.87,30.11) .. controls (77.01,30.11) and (76.32,29.41) .. (76.32,28.56) -- cycle ;
\draw  [fill={rgb, 255:red, 0; green, 0; blue, 0 }  ,fill opacity=1 ] (204.49,27.82) .. controls (204.49,26.96) and (205.19,26.27) .. (206.04,26.27) .. controls (206.9,26.27) and (207.6,26.96) .. (207.6,27.82) .. controls (207.6,28.68) and (206.9,29.37) .. (206.04,29.37) .. controls (205.19,29.37) and (204.49,28.68) .. (204.49,27.82) -- cycle ;

\draw (136.07,20.42) node [anchor=north west][inner sep=0.75pt]    {$-$};
\draw (13,18.4) node [anchor=north west][inner sep=0.75pt]    {$\sigma \ =\ $};

\end{tikzpicture}
\end{center}
With the notation of \cite{Khovanov:2022}, $\sigma$ is given by the cylinder with a \textit{defect circle} (or a \textit{seam}):
\begin{center}
    \tikzset{every picture/.style={line width=0.75pt}} 

\begin{tikzpicture}[x=0.75pt,y=0.75pt,yscale=-1,xscale=1]

\draw  [color={rgb, 255:red, 0; green, 0; blue, 0 }  ,draw opacity=1 ] (57.43,24.17) .. controls (57.43,16.46) and (60,10.22) .. (63.18,10.22) .. controls (66.36,10.22) and (68.93,16.46) .. (68.93,24.17) .. controls (68.93,31.88) and (66.36,38.12) .. (63.18,38.12) .. controls (60,38.12) and (57.43,31.88) .. (57.43,24.17) -- cycle ;
\draw [color={rgb, 255:red, 0; green, 0; blue, 0 }  ,draw opacity=1 ]   (63.18,38.12) -- (122.44,38.12) ;
\draw [color={rgb, 255:red, 0; green, 0; blue, 0 }  ,draw opacity=1 ]   (63.18,10.22) -- (122.44,10.22) ;
\draw  [draw opacity=0][dash pattern={on 0.84pt off 2.51pt}] (120.15,37.91) .. controls (117.93,36.26) and (116.31,30.82) .. (116.31,24.37) .. controls (116.31,17.93) and (117.93,12.5) .. (120.15,10.84) -- (121.58,24.37) -- cycle ; \draw  [dash pattern={on 0.84pt off 2.51pt}] (120.15,37.91) .. controls (117.93,36.26) and (116.31,30.82) .. (116.31,24.37) .. controls (116.31,17.93) and (117.93,12.5) .. (120.15,10.84) ;  
\draw  [draw opacity=0] (121,10.4) .. controls (121.19,10.34) and (121.38,10.32) .. (121.58,10.32) .. controls (124.49,10.32) and (126.85,16.61) .. (126.85,24.37) .. controls (126.85,32.14) and (124.49,38.43) .. (121.58,38.43) .. controls (121.38,38.43) and (121.19,38.4) .. (121,38.35) -- (121.58,24.37) -- cycle ; \draw   (121,10.4) .. controls (121.19,10.34) and (121.38,10.32) .. (121.58,10.32) .. controls (124.49,10.32) and (126.85,16.61) .. (126.85,24.37) .. controls (126.85,32.14) and (124.49,38.43) .. (121.58,38.43) .. controls (121.38,38.43) and (121.19,38.4) .. (121,38.35) ;  

\draw  [draw opacity=0] (92.81,10.22) .. controls (93,10.16) and (93.19,10.14) .. (93.39,10.14) .. controls (96.3,10.14) and (98.66,16.43) .. (98.66,24.19) .. controls (98.66,31.96) and (96.3,38.25) .. (93.39,38.25) .. controls (93.19,38.25) and (93,38.22) .. (92.81,38.17) -- (93.39,24.19) -- cycle ; \draw  [color={rgb, 255:red, 208; green, 2; blue, 27 }  ,draw opacity=1 ] (92.81,10.22) .. controls (93,10.16) and (93.19,10.14) .. (93.39,10.14) .. controls (96.3,10.14) and (98.66,16.43) .. (98.66,24.19) .. controls (98.66,31.96) and (96.3,38.25) .. (93.39,38.25) .. controls (93.19,38.25) and (93,38.22) .. (92.81,38.17) ;  
\draw [color={rgb, 255:red, 208; green, 2; blue, 27 }  ,draw opacity=1 ]   (99.2,24) -- (104.2,24) ;
\draw  [draw opacity=0][dash pattern={on 0.84pt off 2.51pt}] (90.5,36.01) .. controls (88.42,34.46) and (86.89,29.36) .. (86.89,23.32) .. controls (86.89,17.28) and (88.41,12.18) .. (90.49,10.62) -- (91.83,23.32) -- cycle ; \draw  [color={rgb, 255:red, 208; green, 2; blue, 27 }  ,draw opacity=1 ][dash pattern={on 0.84pt off 2.51pt}] (90.5,36.01) .. controls (88.42,34.46) and (86.89,29.36) .. (86.89,23.32) .. controls (86.89,17.28) and (88.41,12.18) .. (90.49,10.62) ;  

\draw (11,14.4) node [anchor=north west][inner sep=0.75pt]    {$\sigma \ =\ $};

\end{tikzpicture}
\end{center}
The short line segment indicates the preferred coorientation of the defect line. The TQFT $\mcF_{h, t}$ together with the involution $\sigma$ coincides with those obtained from the evaluation $\langle \cdot \rangle$ of seamed surfaces given in \cite[Section 3.1]{Khovanov:2022}. See Lemma 3.5 and equations (69) - (75) therein. 

Since $\sigma$ is not a Frobenius algebra isomorphism, it does not induce an involution on the Khovanov complex. Nonetheless, this can be handled by introducing a \textit{twisting} of the Frobenius algebra. For any invertible element $\theta \in A_{h, t}$, the \textit{$\theta$-twisting} of $A_{h, t}$ is the Frobenius algebra $A_{h, t; \theta}$ whose algebra structure is the same as $A_{h, t}$ but the comultiplication and counit maps are twisted as 
\[
    \Delta_\theta(x) := \Delta(\theta^{-1}x),
    \quad 
    \epsilon_\theta(x) := \epsilon(\theta x).
\]
If we consider the $(-1)$-twisting of $A_{h, t}$, the equations of \Cref{prop:sigma-ops} can be rewritten as 
\begin{gather*}
    m \circ (\sigma \otimes \sigma) = \sigma \circ m, \quad
    \iota = \sigma \circ \iota, \\
    \Delta_{(-1)} \circ \sigma = (\sigma \otimes \sigma) \circ \Delta,\quad 
    \epsilon_{(-1)} \circ \sigma = \epsilon
\end{gather*}
which implies that $\sigma$ defines a Frobenius algebra isomorphism
\[
    \sigma\colon A_{h, t} \to A_{h, t; -1}.
\]

For any oriented link diagram $D$, let $\CKh_{h, t; \theta}(D)$ denote the Khovanov complex obtained from $A_{h, t; \theta}$. \Cref{prop:sigma-ops} implies that $\sigma$ induces a chain isomorphism
\[
    \sigma\colon \CKh_{h, t}(D) \to \CKh_{h, t; -1}(D),\quad 
    x \mapsto \sigma^{\otimes k}(x), \quad x\in A^{\otimes k}\subset \CKh_{h, t}(D). 
\]
Furthermore, $\sigma$ acts naturally with respect to link cobordisms. 

\begin{prop}
    Let $S\colon D \to D'$ be a link cobordism represented as a movie of link diagrams, and $\phi_S$ denote the corresponding cobordism map on $\CKh$. Then the following diagram commutes:
    \[
    \begin{tikzcd}
        \CKh_{h, t}(D) \arrow{r}{\phi_S} \arrow{d}{\sigma} & 
        \CKh_{h, t}(D') \arrow{d}{\sigma} \\
        \CKh_{h, t; -1}(D) \arrow{r}{\phi_S} & 
        \CKh_{h, t; -1}(D').
    \end{tikzcd}    
    \]
\end{prop}

\begin{proof}
    Immediate from \Cref{prop:sigma-ops}, since the cobordism map $\phi_S$ is defined by decomposing $S$ into elementary cobordisms and composing the corresponding operations of the Frobenius algebra.
\end{proof}

\begin{remark}

    For an invertible element $\theta \in R_{h, t}$, from \cite[Proposition 3]{Khovanov:2004},\footnote{
        As Ito, Nakagane and Yoshida pointed out in \cite{Ito-Nakagane-Yoshida:2025}, when $\theta$ is an invertible element in $A_{h, t}$, then  a very subtle treatment is required for the construction of a twisting chain isomorphism. 
    } there is a chain isomorphism for any link diagram $D$ 
    \[
        \tau_\theta(D) \colon
        \CKh_{h, t}(D) \to \CKh_{h, t; \theta}(D)
    \]
    corresponding to the $\theta$-twisting of $A_{h, t}$. The above isomorphism $\sigma$ can be composed with the twisting isomorphism 
    \[
        \begin{tikzcd}
            \CKh_{h, t}(D) \arrow{r}{\sigma}& 
            \CKh_{h, t; -1}(D) \arrow{r}{\tau_{-1}(D)}&
            \CKh_{h, t}(D)
        \end{tikzcd}
    \]
    so that the composition is an involution on $\CKh_{h, t}(D)$. However, the twisting isomorphism $\tau_\theta(D)$ is not natural with respect to the Reidemeister moves (at least if we take the same maps as in \cite{BarNatan:2005}), as the following example shows. 
\end{remark}

\begin{ex}
    Consider diagrams $D, D'$ that are related by a single R2 move, as
    \begin{center}
        \tikzset{every picture/.style={line width=0.75pt}} 

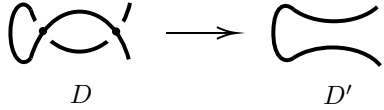
\begin{tikzpicture}[x=0.75pt,y=0.75pt,yscale=-.9,xscale=.9]

\draw    (94.57,21.96) -- (131.95,21.96) ;
\draw [shift={(133.95,21.96)}, rotate = 180] [color={rgb, 255:red, 0; green, 0; blue, 0 }  ][line width=0.75]    (10.93,-3.29) .. controls (6.95,-1.4) and (3.31,-0.3) .. (0,0) .. controls (3.31,0.3) and (6.95,1.4) .. (10.93,3.29)   ;
\draw  [draw opacity=0][line width=1.5]  (74.33,5.06) .. controls (70.9,20.92) and (59.32,32.55) .. (45.68,32.41) .. controls (33.38,32.28) and (22.94,22.62) .. (18.74,9.04) -- (46.05,-4.49) -- cycle ; \draw  [line width=1.5]  (74.33,5.06) .. controls (70.9,20.92) and (59.32,32.55) .. (45.68,32.41) .. controls (33.38,32.28) and (22.94,22.62) .. (18.74,9.04) ;  
\draw  [draw opacity=0][fill={rgb, 255:red, 255; green, 255; blue, 255 }  ,fill opacity=1 ] (73.58,21.4) .. controls (73.55,25.37) and (70.3,28.57) .. (66.32,28.54) .. controls (62.35,28.51) and (59.15,25.26) .. (59.18,21.28) .. controls (59.21,17.3) and (62.46,14.11) .. (66.44,14.14) .. controls (70.42,14.17) and (73.61,17.42) .. (73.58,21.4) -- cycle ;
\draw  [draw opacity=0][fill={rgb, 255:red, 0; green, 0; blue, 0 }  ,fill opacity=1 ] (68.84,21.15) .. controls (68.83,22.42) and (67.8,23.44) .. (66.53,23.43) .. controls (65.25,23.42) and (64.23,22.38) .. (64.24,21.11) .. controls (64.25,19.84) and (65.29,18.82) .. (66.56,18.83) .. controls (67.83,18.84) and (68.85,19.87) .. (68.84,21.15) -- cycle ;

\draw  [draw opacity=0][fill={rgb, 255:red, 255; green, 255; blue, 255 }  ,fill opacity=1 ] (32.59,21.15) .. controls (32.56,25.12) and (29.31,28.32) .. (25.33,28.29) .. controls (21.36,28.26) and (18.16,25.01) .. (18.19,21.03) .. controls (18.22,17.05) and (21.47,13.86) .. (25.45,13.89) .. controls (29.42,13.92) and (32.62,17.17) .. (32.59,21.15) -- cycle ;
\draw  [draw opacity=0][fill={rgb, 255:red, 0; green, 0; blue, 0 }  ,fill opacity=1 ] (27.85,20.9) .. controls (27.84,22.17) and (26.81,23.19) .. (25.53,23.18) .. controls (24.26,23.17) and (23.24,22.13) .. (23.25,20.86) .. controls (23.26,19.59) and (24.3,18.57) .. (25.57,18.58) .. controls (26.84,18.59) and (27.86,19.63) .. (27.85,20.9) -- cycle ;

\draw  [draw opacity=0][line width=1.5]  (18.71,36.1) .. controls (22.97,20.78) and (33.71,9.92) .. (46.24,9.98) .. controls (58.88,10.04) and (69.6,21.18) .. (73.64,36.75) -- (46.05,50.6) -- cycle ; \draw  [line width=1.5]  (18.71,36.1) .. controls (22.97,20.78) and (33.71,9.92) .. (46.24,9.98) .. controls (58.88,10.04) and (69.6,21.18) .. (73.64,36.75) ;  
\draw [line width=1.5]    (18.74,9.04) .. controls (14.67,2.28) and (7.73,6.95) .. (7.33,20.95) .. controls (6.93,34.95) and (13.67,42.28) .. (18.71,36.1) ;
\draw  [draw opacity=0][line width=1.5]  (166.04,35.89) .. controls (171.43,32.29) and (179.3,30.04) .. (188.05,30.08) .. controls (199.81,30.13) and (209.93,34.3) .. (214.57,40.27) -- (187.98,47.53) -- cycle ; \draw  [line width=1.5]  (166.04,35.89) .. controls (171.43,32.29) and (179.3,30.04) .. (188.05,30.08) .. controls (199.81,30.13) and (209.93,34.3) .. (214.57,40.27) ;  
\draw [line width=1.5]    (166.07,8.84) .. controls (161,4.28) and (155.07,6.74) .. (154.67,20.74) .. controls (154.27,34.74) and (160.33,39.95) .. (166.04,35.89) ;
\draw  [draw opacity=0][line width=1.5]  (166.07,8.84) .. controls (171.43,12.81) and (179.45,15.35) .. (188.42,15.39) .. controls (198.65,15.44) and (207.67,12.22) .. (212.94,7.31) -- (188.51,-2.8) -- cycle ; \draw  [line width=1.5]  (166.07,8.84) .. controls (171.43,12.81) and (179.45,15.35) .. (188.42,15.39) .. controls (198.65,15.44) and (207.67,12.22) .. (212.94,7.31) ;  

\draw (181,49.16) node [anchor=north west][inner sep=0.75pt]    {$D'$};
\draw (39,49.16) node [anchor=north west][inner sep=0.75pt]    {$D$};

\end{tikzpicture}
    \end{center}
    The chain homotopy equivalence $F$ on the corresponding chain complexes is given by 
    \begin{center}
        \input{tikzpictures/example-sigma-expand}
    \end{center}    
    Here, to save space, the untwisted complex $\CKh_{h, t}$ and the twisted complex $\CKh_{h, t; \theta}$ are drawn together, and the action of the twisting isomorphism $\tau_\theta$ on each vertex is indicated by the blue loop. By focusing on the $01$-component of $\CKh_{h, t}(D)$, one can see that the following diagram does not commute
    \[
        \begin{tikzcd}
            \CKh_{h, t}(D) \arrow{d}{F} \arrow{r}{\tau_\theta(D)} & 
            \CKh_{h, t; \theta}(D) \arrow{d}{F} \\
            \CKh_{h, t}(D') \arrow{r}{\tau_\theta(D')} & 
            \CKh_{h, t; \theta}(D').
        \end{tikzcd}
    \]
\end{ex}

\begin{question}
    Is it possible to adjust the cobordism maps for the $\theta$-twisted complex $\CKh_{h, t; \theta}$ so that the twisting isomorphisms can be collectively regarded as a natural transformation? In other words, can we make the following diagram commute up to chain homotopy?
    \[
\begin{tikzcd}[row sep=3em, column sep=4em]
\CKh_{h, t}(D) 
    \arrow[r, "\tau_\theta(D)"] 
    \arrow[d, "\CKh_{h, t}(S)"] 
& \CKh_{h, t; \theta}(D) 
    \arrow[d, "\CKh_{h, t; \theta}(S)"] \\
\CKh_{h, t}(D') 
    \arrow[r, "\tau_\theta(D')"] 
& \CKh_{h, t; \theta}(D')
\end{tikzcd}
    \]
\end{question}

\begin{remark}
    If we work over $\F_2$ by tensoring it to the ground ring $R_{h, t}$, then $\sigma$ becomes a Frobenius algebra involution and the induced $\sigma$ gives an involution on $\CKh_{h, t}(-; \F_2)$. In \cite{Chen-Yang:2025}, Chen and Yang studies the \textit{(intrinsic) involutive Khovanov homology}, defined as the homology of the mapping cone of $\id + \sigma$ over $\F_2$. 
\end{remark}

\subsection{Graded involution \texorpdfstring{$\hsigma$}{sigma-hat}}
\label{subsec:hsigma}

The sign inconsistency of $\sigma$ with respect to the operations of $A_{h, t}$ can be fixed by modifying the definition of $\sigma$. Define a pair of ring involutions, $\hsigma_0$ on $R_{h, t}$ and $\hsigma_1$ on $A$ by 
\[
    \hsigma_0(r) := (-1)^{\frac{\deg(r)}{2}} r,\quad
    \hsigma_1(x) := (-1)^{\frac{\deg(x)}{2}} \sigma(x).
\]
for $r \in R_{h, t}$ and $x \in A_{h, t}$. Now, we have 
\[
    \hsigma_1(rx) = \hsigma_0(r) \hsigma_1(x)
\]
and in particular, 
\[
    \hsigma_0(1) = 1, \ \ \hsigma_0(h)=-h, \ \ \hsigma_0(t)=t, \ \ \hsigma_1(X) = X - h = Y. 
\]
Note that $\hsigma_1$ is not an $R_{h, t}$-module homomorphism on $A_{h,t}$ unlike our original involution $\sigma$. Although this might look unnatural, we have the following:

\begin{prop}
\label{prop:hsigma-ops}
    The pair of involutions $(\hsigma_0, \hsigma_1)$ satisfies
    \begin{gather*}
        m \circ (\hsigma_1 \otimes \hsigma_1) = \hsigma_1 \circ m, \quad
        \iota \circ \hsigma_0 = \hsigma_1 \circ \iota, \\
        \Delta \circ \hsigma_1 = (\hsigma_1 \otimes \hsigma_1) \circ \Delta,\quad 
        \epsilon \circ \hsigma_1 = \hsigma_0 \circ \epsilon.
    \end{gather*}
\end{prop}

\begin{proof}
    Immediate from \Cref{prop:sigma-ops} and the degrees of the operations
    \[
        \deg(m) = \deg(\iota) = 0, \quad \deg(\Delta) = 2, \quad \deg(\epsilon) = -2. \qedhere
    \]
\end{proof}

Consider the subring $R'_{h, t} := \Z[h^2, t]$ of $R_{h, t}$, which is supported in degrees $0 \bmod{4}$. $R_{h, t}$ is a free graded $R'_{h, t}$-module of rank $2$, generated by $1$ and $h$. The involution $\hsigma_0$ on $R_{h, t}$ is an $R'_{h, t}$-module endomorphism, represented by the matrix 
\[
    \begin{pmatrix}
        1 & \\
        & -1 
    \end{pmatrix}.
\]
Similarly, $A_{h, t}$ is a rank $4$ free $R'_{h, t}$-module, generated by $1, h, X, hX$, and the involution $\hsigma_1$ on $A_{h, t}$ is an $R'_{h, t}$-module endomorphism represented by
\[
    \begin{pmatrix}
        1 & & & h^2\\
        & -1 & -1 \\
        & & 1 \\
        & & & -1
    \end{pmatrix}.
\]
With $U := 2X - h$, one can see that $\hsigma$ restricts to $\id$ on the $R'_{h, t}$-submodule $R'_{h, t}\langle1, U\rangle  \subset A_{h, t}$ and to $-\id$ on $R'_{h, t}\langle h, hU\rangle \subset A_{h, t}$. Multiplication by $h$ is a non-invertible map between these two submodules. If we adjoin $2^{-1}$ to the rings (while denoting them by the same symbols), then $\{1, U, h, hU\}$ form a basis for $A_{h, t}$ over $R'_{h, t}$, giving it an eigendecomposition into $(+1)$-eigenspace $R'_{h, t}\langle1, U\rangle$ and $(-1)$-eigenspace $R'_{h, t}\langle h, hU\rangle$.

Hereafter, we omit the subscripts from $\hsigma_0$ and $\hsigma_1$ when there is no confusion. Moreover, we extend the involution over arbitrary $r$-fold tensor product of $A_{h,t}$ $(r \geq 0)$, and denote it by the same symbol
\[
    \hsigma := \hsigma_1^{\otimes r}\colon A_{h,t}^{\otimes r} \to A_{h,t}^{\otimes r}.
\]

The involution $\hsigma$ can be interpreted as inserting a \textit{defect plane} $\Pi$ in $\R^2 \times I$ perpendicular to the time axis $I$. Suppose $S$ is a cobordism in $\R^2 \times I$ that intersects $\Pi$ transversely. The equation $\hsigma(X) = Y$ can be interpreted as follows: if a dot $\bdot$ on $S$ passes $\Pi$, then it turns into $\hdot$.
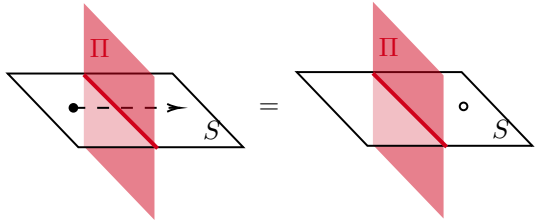
\begin{center}
    \tikzset{every picture/.style={line width=0.75pt}} 

\begin{tikzpicture}[x=0.75pt,y=0.75pt,yscale=-.6,xscale=.6]

\draw  [draw opacity=0][fill={rgb, 255:red, 208; green, 2; blue, 27 }  ,fill opacity=0.5 ] (71.5,130.21) -- (71.5,70) -- (130.9,132) -- (130.9,192.21) -- cycle ;
\draw  [fill={rgb, 255:red, 255; green, 255; blue, 255 }  ,fill opacity=0.5 ] (146.1,69) -- (7.5,69) -- (66.9,131) -- (205.5,131) -- cycle ;
\draw [color={rgb, 255:red, 208; green, 2; blue, 27 }  ,draw opacity=1 ][line width=1.5]    (71.5,70) -- (133.5,132) ;
\draw  [fill={rgb, 255:red, 0; green, 0; blue, 0 }  ,fill opacity=1 ] (59.54,97.86) .. controls (59.54,96.09) and (60.97,94.66) .. (62.74,94.66) .. controls (64.51,94.66) and (65.95,96.09) .. (65.95,97.86) .. controls (65.95,99.63) and (64.51,101.07) .. (62.74,101.07) .. controls (60.97,101.07) and (59.54,99.63) .. (59.54,97.86) -- cycle ;
\draw  [draw opacity=0][fill={rgb, 255:red, 208; green, 2; blue, 27 }  ,fill opacity=0.5 ] (314.5,129) -- (314.5,68.79) -- (373.9,130.79) -- (373.9,191) -- cycle ;
\draw  [fill={rgb, 255:red, 255; green, 255; blue, 255 }  ,fill opacity=0.5 ] (389.1,67.79) -- (250.5,67.79) -- (309.9,129.79) -- (448.5,129.79) -- cycle ;
\draw  [draw opacity=0][fill={rgb, 255:red, 208; green, 2; blue, 27 }  ,fill opacity=0.5 ] (314.5,68.79) -- (314.5,8.59) -- (373.9,70.59) -- (373.9,130.79) -- cycle ;
\draw [color={rgb, 255:red, 208; green, 2; blue, 27 }  ,draw opacity=1 ][line width=1.5]    (314.5,68.79) -- (376.5,130.79) ;
\draw  [color={rgb, 255:red, 0; green, 0; blue, 0 }  ,draw opacity=1 ][fill={rgb, 255:red, 255; green, 255; blue, 255 }  ,fill opacity=1 ][line width=0.75]  (387.5,96.66) .. controls (387.5,94.89) and (388.93,93.45) .. (390.7,93.45) .. controls (392.47,93.45) and (393.91,94.89) .. (393.91,96.66) .. controls (393.91,98.43) and (392.47,99.86) .. (390.7,99.86) .. controls (388.93,99.86) and (387.5,98.43) .. (387.5,96.66) -- cycle ;
\draw  [dash pattern={on 4.5pt off 4.5pt}]  (62.74,97.86) -- (150.5,97.86) ;
\draw [shift={(152.5,97.86)}, rotate = 180] [color={rgb, 255:red, 0; green, 0; blue, 0 }  ][line width=0.75]    (10.93,-3.29) .. controls (6.95,-1.4) and (3.31,-0.3) .. (0,0) .. controls (3.31,0.3) and (6.95,1.4) .. (10.93,3.29)   ;
\draw  [draw opacity=0][fill={rgb, 255:red, 208; green, 2; blue, 27 }  ,fill opacity=0.5 ] (71.5,70) -- (71.5,9.79) -- (130.9,71.79) -- (130.9,132) -- cycle ;

\draw (179.22,126.6) node [anchor=south] [inner sep=0.75pt]    {$S$};
\draw (74,37.4) node [anchor=north west][inner sep=0.75pt]  [color={rgb, 255:red, 208; green, 2; blue, 27 }  ,opacity=1 ]  {$\Pi $};
\draw (422.22,125.39) node [anchor=south] [inner sep=0.75pt]    {$S$};
\draw (317,36.19) node [anchor=north west][inner sep=0.75pt]  [color={rgb, 255:red, 208; green, 2; blue, 27 }  ,opacity=1 ]  {$\Pi $};
\draw (216,91.4) node [anchor=north west][inner sep=0.75pt]    {$=$};

\end{tikzpicture}
\end{center}
In particular, if a sphere with a single dot passes $\Pi$, it turns into a sphere with a hollow dot, which recovers $\hsigma(1) = 1$. If a sphere two dots passes $\Pi$, it turns into a sphere with two hollow dots, which recovers $\hsigma(h) = -h$. The equation $\hsigma(t) = t$ can be given a similar interpretation.
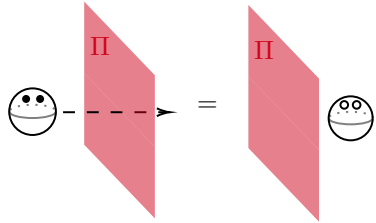
\begin{center}
    \tikzset{every picture/.style={line width=0.75pt}} 

\begin{tikzpicture}[x=0.75pt,y=0.75pt,yscale=-.6,xscale=.6]

\draw  [draw opacity=0][fill={rgb, 255:red, 208; green, 2; blue, 27 }  ,fill opacity=0.5 ] (71.5,130.21) -- (71.5,70) -- (130.9,132) -- (130.9,192.21) -- cycle ;
\draw  [draw opacity=0][fill={rgb, 255:red, 208; green, 2; blue, 27 }  ,fill opacity=0.5 ] (209.5,133) -- (209.5,72.79) -- (268.9,134.79) -- (268.9,195) -- cycle ;
\draw  [draw opacity=0][fill={rgb, 255:red, 208; green, 2; blue, 27 }  ,fill opacity=0.5 ] (209.5,72.79) -- (209.5,12.59) -- (268.9,74.59) -- (268.9,134.79) -- cycle ;
\draw  [dash pattern={on 4.5pt off 4.5pt}]  (53.74,102.86) -- (141.5,102.86) ;
\draw [shift={(143.5,102.86)}, rotate = 180] [color={rgb, 255:red, 0; green, 0; blue, 0 }  ][line width=0.75]    (10.93,-3.29) .. controls (6.95,-1.4) and (3.31,-0.3) .. (0,0) .. controls (3.31,0.3) and (6.95,1.4) .. (10.93,3.29)   ;
\draw  [draw opacity=0][fill={rgb, 255:red, 208; green, 2; blue, 27 }  ,fill opacity=0.5 ] (71.5,70) -- (71.5,9.79) -- (130.9,71.79) -- (130.9,132) -- cycle ;
\draw  [draw opacity=0] (47.51,102.83) .. controls (46.72,105.58) and (38.33,107.75) .. (28.11,107.75) .. controls (17.36,107.75) and (8.65,105.36) .. (8.65,102.41) .. controls (8.65,102.23) and (8.68,102.05) .. (8.74,101.87) -- (28.11,102.41) -- cycle ; \draw  [color={rgb, 255:red, 128; green, 128; blue, 128 }  ,draw opacity=1 ] (47.51,102.83) .. controls (46.72,105.58) and (38.33,107.75) .. (28.11,107.75) .. controls (17.36,107.75) and (8.65,105.36) .. (8.65,102.41) .. controls (8.65,102.23) and (8.68,102.05) .. (8.74,101.87) ;  
\draw  [draw opacity=0][dash pattern={on 0.84pt off 2.51pt}] (9.13,101.08) .. controls (10.68,98.56) and (18.68,96.64) .. (28.32,96.64) .. controls (38.91,96.64) and (47.53,98.96) .. (47.77,101.85) -- (28.32,101.98) -- cycle ; \draw  [color={rgb, 255:red, 128; green, 128; blue, 128 }  ,draw opacity=1 ][dash pattern={on 0.84pt off 2.51pt}] (9.13,101.08) .. controls (10.68,98.56) and (18.68,96.64) .. (28.32,96.64) .. controls (38.91,96.64) and (47.53,98.96) .. (47.77,101.85) ;  
\draw   (8.5,102.41) .. controls (8.5,91.58) and (17.28,82.8) .. (28.11,82.8) .. controls (38.94,82.8) and (47.72,91.58) .. (47.72,102.41) .. controls (47.72,113.24) and (38.94,122.01) .. (28.11,122.01) .. controls (17.28,122.01) and (8.5,113.24) .. (8.5,102.41) -- cycle ;
\draw  [fill={rgb, 255:red, 0; green, 0; blue, 0 }  ,fill opacity=1 ] (20.16,91.71) .. controls (20.16,90.32) and (21.29,89.2) .. (22.68,89.2) .. controls (24.07,89.2) and (25.19,90.32) .. (25.19,91.71) .. controls (25.19,93.1) and (24.07,94.22) .. (22.68,94.22) .. controls (21.29,94.22) and (20.16,93.1) .. (20.16,91.71) -- cycle ;
\draw  [fill={rgb, 255:red, 0; green, 0; blue, 0 }  ,fill opacity=1 ] (31.93,91.71) .. controls (31.93,90.32) and (33.06,89.2) .. (34.45,89.2) .. controls (35.83,89.2) and (36.96,90.32) .. (36.96,91.71) .. controls (36.96,93.1) and (35.83,94.22) .. (34.45,94.22) .. controls (33.06,94.22) and (31.93,93.1) .. (31.93,91.71) -- cycle ;
\draw  [draw opacity=0] (313.81,108.4) .. controls (313.07,111) and (305.16,113.04) .. (295.51,113.04) .. controls (285.37,113.04) and (277.15,110.78) .. (277.15,108) .. controls (277.15,107.83) and (277.18,107.66) .. (277.24,107.5) -- (295.51,108) -- cycle ; \draw  [color={rgb, 255:red, 128; green, 128; blue, 128 }  ,draw opacity=1 ] (313.81,108.4) .. controls (313.07,111) and (305.16,113.04) .. (295.51,113.04) .. controls (285.37,113.04) and (277.15,110.78) .. (277.15,108) .. controls (277.15,107.83) and (277.18,107.66) .. (277.24,107.5) ;  
\draw  [draw opacity=0][dash pattern={on 0.84pt off 2.51pt}] (277.6,106.75) .. controls (279.07,104.37) and (286.62,102.56) .. (295.71,102.56) .. controls (305.7,102.56) and (313.83,104.75) .. (314.06,107.48) -- (295.71,107.6) -- cycle ; \draw  [color={rgb, 255:red, 128; green, 128; blue, 128 }  ,draw opacity=1 ][dash pattern={on 0.84pt off 2.51pt}] (277.6,106.75) .. controls (279.07,104.37) and (286.62,102.56) .. (295.71,102.56) .. controls (305.7,102.56) and (313.83,104.75) .. (314.06,107.48) ;  
\draw   (277.01,108) .. controls (277.01,97.78) and (285.29,89.5) .. (295.51,89.5) .. controls (305.73,89.5) and (314.01,97.78) .. (314.01,108) .. controls (314.01,118.22) and (305.73,126.5) .. (295.51,126.5) .. controls (285.29,126.5) and (277.01,118.22) .. (277.01,108) -- cycle ;
\draw  [color={rgb, 255:red, 0; green, 0; blue, 0 }  ,draw opacity=1 ][fill={rgb, 255:red, 255; green, 255; blue, 255 }  ,fill opacity=1 ] (286.91,96.7) .. controls (286.91,94.93) and (288.34,93.5) .. (290.11,93.5) .. controls (291.88,93.5) and (293.31,94.93) .. (293.31,96.7) .. controls (293.31,98.47) and (291.88,99.91) .. (290.11,99.91) .. controls (288.34,99.91) and (286.91,98.47) .. (286.91,96.7) -- cycle ;
\draw  [color={rgb, 255:red, 0; green, 0; blue, 0 }  ,draw opacity=1 ][fill={rgb, 255:red, 255; green, 255; blue, 255 }  ,fill opacity=1 ] (297.31,96.7) .. controls (297.31,94.93) and (298.74,93.5) .. (300.51,93.5) .. controls (302.28,93.5) and (303.71,94.93) .. (303.71,96.7) .. controls (303.71,98.47) and (302.28,99.91) .. (300.51,99.91) .. controls (298.74,99.91) and (297.31,98.47) .. (297.31,96.7) -- cycle ;

\draw (74,37.4) node [anchor=north west][inner sep=0.75pt]  [color={rgb, 255:red, 208; green, 2; blue, 27 }  ,opacity=1 ]  {$\Pi $};
\draw (164,91.4) node [anchor=north west][inner sep=0.75pt]    {$=$};
\draw (212,40.19) node [anchor=north west][inner sep=0.75pt]  [color={rgb, 255:red, 208; green, 2; blue, 27 }  ,opacity=1 ]  {$\Pi $};

\end{tikzpicture}
\end{center}
Each equation of \Cref{prop:hsigma-ops} can be given pictorial descriptions,
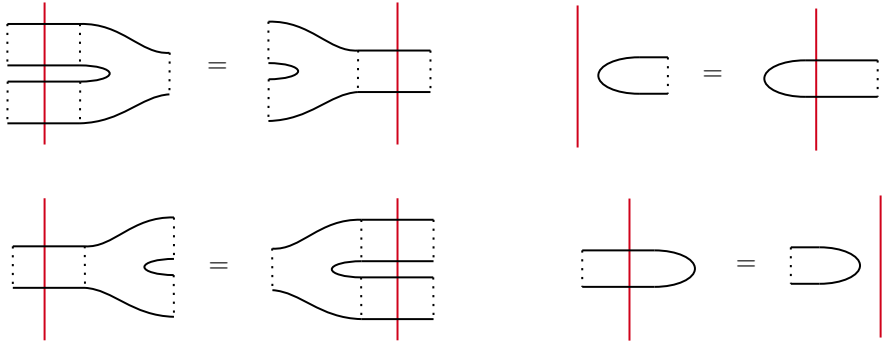
\begin{center}
    \input{tikzpictures/hsigma-frob-ops-picture}
\end{center}
where boundary circles of cobordism surfaces are shown by dashed intervals. 
Thus, instead of isotoping $S$, we may freely move $\Pi$ along the time axis without making any change to the underlying surface of $S$, while swapping the dots $\bdot$ and $\hdot$ on $S$ as $\Pi$ pass by. If two such parallel defect planes meet, they can be canceled, since $\hsigma^2=1$.
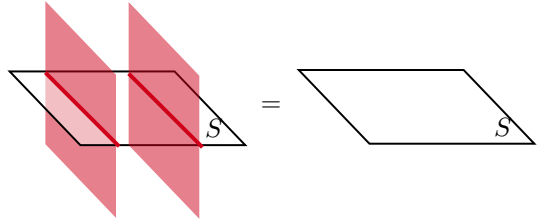
\begin{center}
    \tikzset{every picture/.style={line width=0.75pt}} 

\begin{tikzpicture}[x=0.75pt,y=0.75pt,yscale=-.6,xscale=.6]

\draw  [draw opacity=0][fill={rgb, 255:red, 208; green, 2; blue, 27 }  ,fill opacity=0.5 ] (37.5,130.21) -- (37.5,70) -- (96.9,132) -- (96.9,192.21) -- cycle ;
\draw  [fill={rgb, 255:red, 255; green, 255; blue, 255 }  ,fill opacity=0.5 ] (146.1,69) -- (7.5,69) -- (66.9,131) -- (205.5,131) -- cycle ;
\draw [color={rgb, 255:red, 208; green, 2; blue, 27 }  ,draw opacity=1 ][line width=1.5]    (37.5,70) -- (99.5,132) ;
\draw  [fill={rgb, 255:red, 255; green, 255; blue, 255 }  ,fill opacity=0.5 ] (389.1,67.79) -- (250.5,67.79) -- (309.9,129.79) -- (448.5,129.79) -- cycle ;
\draw  [draw opacity=0][fill={rgb, 255:red, 208; green, 2; blue, 27 }  ,fill opacity=0.5 ] (37.5,70) -- (37.5,9.79) -- (96.9,71.79) -- (96.9,132) -- cycle ;
\draw  [draw opacity=0][fill={rgb, 255:red, 208; green, 2; blue, 27 }  ,fill opacity=0.5 ] (107.5,131) -- (107.5,70.79) -- (166.9,132.79) -- (166.9,193) -- cycle ;
\draw [color={rgb, 255:red, 208; green, 2; blue, 27 }  ,draw opacity=1 ][line width=1.5]    (107.5,70.79) -- (169.5,132.79) ;
\draw  [draw opacity=0][fill={rgb, 255:red, 208; green, 2; blue, 27 }  ,fill opacity=0.5 ] (107.5,70.79) -- (107.5,10.59) -- (166.9,72.59) -- (166.9,132.79) -- cycle ;

\draw (179.22,126.6) node [anchor=south] [inner sep=0.75pt]    {$S$};
\draw (422.22,125.39) node [anchor=south] [inner sep=0.75pt]    {$S$};
\draw (216,91.4) node [anchor=north west][inner sep=0.75pt]    {$=$};

\end{tikzpicture}
\end{center}
Using (NC), any intersecting circle of $\Pi$ and $S$ can be resolved as follows 
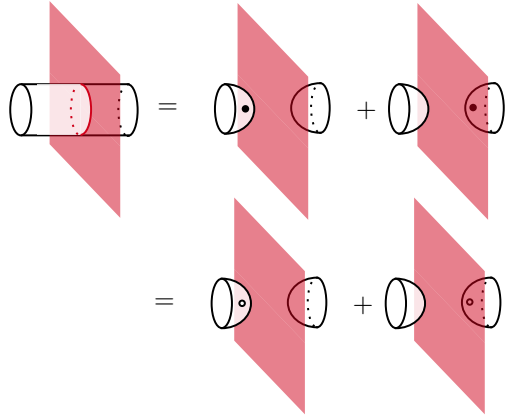
\begin{center}
    \input{tikzpictures/defect-neck-cut}
\end{center}
More generally, one may consider a \textit{defect surface} $\Sigma$, which is an oriented (possibly disconnected) surface embedded in the interior of $\R^2 \times I$. A closed component of $\Sigma$ can be shrunk to a point and be removed.

Formally, we define an \textit{involutive Frobenius extension} to be an Frobenius extension $(R, A)$ equipped with a pair of ring involutions $\hsigma_0$ on $R$ and $\hsigma_1$ on $A$ satisfying the equations of \Cref{prop:hsigma-ops}. \textit{Homomorphisms} and \textit{isomorphisms} of involutive Frobenius extensions are those of Frobenius extensions that also commute with the involutions. 

\begin{remark}
More generally, given a commutative Frobenius extension $(R, A)$ and an automorphism $\psi$ of the pair $(R, A)$ preserving Frobenius structure, one can introduce $\psi$-hyperplanes into 2D cobordisms. It is then convenient to assume that 1-manifolds live in $\R^n$ for $n\ge 4$ and 2-cobordisms live in $\R^n\times [0,1]$, to avoid possible knottedness of 1-manifolds and their cobordisms and represent hyperplanes as $\R^n\times \{y\}$ for $0<y<1$. Alternatively, one can consider the case $n=1$ and work with 1-manifolds embedded in the plane and cobordisms between them in $\R^2\times [0,1]$. In the latter case, however, there should exist more complicated TQFTs for this cobordism category, where one takes into account how circles are nested in the plane, and likewise for cobordisms.  
\end{remark}
 
Now, define an involution $\hsigma$ on $\CKh_{h, t}(D)$ by 
\[
    \hsigma\colon \CKh_{h, t}(D) \to \CKh_{h, t}(D);\quad
    x \mapsto (\hsigma_1 \otimes \cdots \otimes \hsigma_1)(x), \ \ \mathrm{for} \ \ x\in A_{h,t}^{\otimes k}\subset \CKh_{h, t}(D). 
\]
Again, note that $\hsigma$ of the complex $\CKh_{h, t}(D)$ is not an $R_{h, t}$-module involution, but is an $R'_{h, t}$-module involution. 

\begin{prop}
\label{prop:cobordism-map-and-hsigma-commute}
    Let $S\colon D \to D'$ be a link cobordism represented as a movie of link diagrams, and $\phi_S$ denote the corresponding cobordism map on $\CKh_{h, t}$. Then the following diagram commutes:
    \[
    \begin{tikzcd}
        \CKh_{h, t}(D) 
            \arrow{r}{\phi_S} 
            \arrow{d}{\hsigma} & 
        \CKh_{h, t}(D') 
            \arrow{d}{\hsigma} \\
        \CKh_{h, t}(D) 
            \arrow{r}{\phi_S} & 
        \CKh_{h, t}(D').
    \end{tikzcd}    
    \]
\end{prop}

\begin{proof}
    As described in \cite{BarNatan:2005}, the Reidemeister-move maps (Figures 5, 6, 9 therein) and Morse-move maps (in Section 8.1) are all described by undotted cobordisms. Thus the commutativity is immediate from the fact that a defect plane can pass through any undotted cobordisms without changing them.
\end{proof}

The following proposition is also immediate from the definition of $\hsigma$. 

\begin{prop}
    For link diagrams $D, D'$, let $D \sqcup D'$ denote their  disjoint union. The involution $\hsigma$ commutes with the canonical isomorphism computing the chain complex for the disjoint union: 
    \[
\begin{tikzcd}
\CKh_{h, t}(D) \otimes \CKh_{h, t}(D') \arrow[r, "\cong"]  \arrow[d, "\hsigma"'] & \CKh_{h, t}(D \sqcup D') \arrow[d, "\hsigma"] \\
\CKh_{h, t}(D) \otimes \CKh_{h, t}(D') \arrow[r, "\cong"] & \CKh_{h, t}(D \sqcup D').
\end{tikzcd}
    \]
\end{prop}

\subsection{Duality}

Consider the non-degenerate pairing on $A_{h, t}$
\[
    \beta := \epsilon \circ m\colon A_{h, t} \otimes A_{h, t} \to R_{h, t}.
\]
Let $\dual$ denote the associated isomorphism
\[
    \dual\colon A_{h, t} \to A^*_{h, t}
\]
where $A^*_{h, t} := \Hom_{R_{h, t}}(A_{h, t}, R_{h, t})$ is the dual $R_{h, t}$-module of $A_{h, t}$. We call $\{\dual(1), \dual(X)\}$ the \textit{standard basis} of $A^*_{h, t}$. Its elements have a cobordism description:
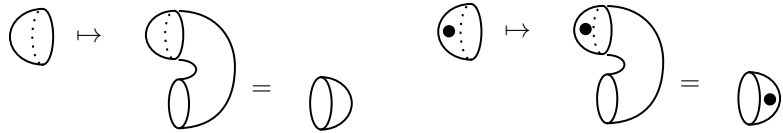
\begin{center}
    \input{tikzpictures/duality-map}
\end{center}
With $Y = X - h$, one can see that $\{1, X\}$ and $\{Y, 1\}$ are mutually dual with respect to the non-degenerate pairing $\beta$. Thus the standard basis $\{\dual(1), \dual(X)\}$ for $ A^*_{h, t}$ is precisely the (algebraic) dual basis of $\{Y, 1\}$ for $A_{h, t}$. 

From \Cref{prop:hsigma-ops}, we have
\[
    \beta \circ (\hsigma_1 \otimes \hsigma_1) = \hsigma_0 \circ \beta 
\]
which gives, for any $x \in A_{h,t}$, 
\[
    \dual(\hsigma_1(x)) \circ \hsigma_1 = \hsigma_0 \circ \dual(x).
\]
This can be visualized as follows:
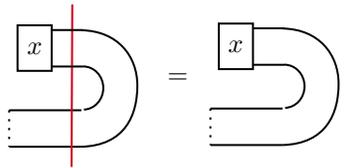
\begin{center}
    \tikzset{every picture/.style={line width=0.75pt}} 

\begin{tikzpicture}[x=0.75pt,y=0.75pt,yscale=-1,xscale=1]

\draw    (53.43,62.46) -- (17.03,62.46) ;
\draw    (53.43,80.87) -- (17.03,80.87) ;

\draw  [dash pattern={on 0.84pt off 2.51pt}]  (17.03,62.46) -- (17.03,80.87) ;
\draw  [draw opacity=0] (54.71,62.12) .. controls (60.15,61.92) and (64.5,57.12) .. (64.5,51.23) .. controls (64.5,45.21) and (59.96,40.33) .. (54.36,40.33) -- (54.36,51.23) -- cycle ; \draw   (54.71,62.12) .. controls (60.15,61.92) and (64.5,57.12) .. (64.5,51.23) .. controls (64.5,45.21) and (59.96,40.33) .. (54.36,40.33) ;  
\draw    (53.43,22.06) .. controls (71.88,22.39) and (81.68,35.23) .. (81.67,49.95) .. controls (81.67,64.68) and (75.69,80.22) .. (53.43,80.87) ;

\draw [color={rgb, 255:red, 208; green, 2; blue, 27 }  ,draw opacity=1 ]   (48.86,9.06) -- (48.5,91.06) ;
\draw    (54.43,21.96) -- (38.25,21.96) ;
\draw    (54.43,40.37) -- (38.25,40.37) ;

\draw   (21.25,19.56) -- (38.5,19.56) -- (38.5,42.56) -- (21.25,42.56) -- cycle ;
\draw    (154.93,61.59) -- (118.53,61.59) ;
\draw    (154.93,80) -- (118.53,80) ;

\draw  [dash pattern={on 0.84pt off 2.51pt}]  (118.53,61.59) -- (118.53,80) ;
\draw  [draw opacity=0] (156.21,61.24) .. controls (161.65,61.04) and (166,56.24) .. (166,50.35) .. controls (166,44.33) and (161.46,39.45) .. (155.86,39.45) -- (155.86,50.35) -- cycle ; \draw   (156.21,61.24) .. controls (161.65,61.04) and (166,56.24) .. (166,50.35) .. controls (166,44.33) and (161.46,39.45) .. (155.86,39.45) ;  
\draw    (154.93,21.19) .. controls (173.38,21.51) and (183.18,34.35) .. (183.17,49.08) .. controls (183.17,63.8) and (177.19,79.35) .. (154.93,80) ;

\draw [color={rgb, 255:red, 208; green, 2; blue, 27 }  ,draw opacity=1 ]   (191.36,8.69) -- (191,90.69) ;
\draw    (155.93,21.09) -- (139.75,21.09) ;
\draw    (155.93,39.5) -- (139.75,39.5) ;

\draw   (122.75,18.69) -- (140,18.69) -- (140,41.69) -- (122.75,41.69) -- cycle ;

\draw (29.88,31.06) node    {$x$};
\draw (95.5,42.09) node [anchor=north west][inner sep=0.75pt]    {$=$};
\draw (131.38,30.19) node    {$x$};

\end{tikzpicture}
\end{center}

We define an involution $\hsigma_\dual$ on $A_{h,t}^*$ by
\[
    \hsigma_\dual\colon A_{h,t}^* \to A_{h,t}^*,\quad 
    f \mapsto \hsigma_0 \circ f \circ \hsigma_1.
\]
Then we have the following commutative diagram:
\[
    \begin{tikzcd}[row sep=3em, column sep=3em]
        A_{h,t} \arrow[r, "\dual"] \arrow[d, "\hsigma_1"'] 
        & A_{h,t}^* \arrow[d, "\hsigma_\dual"] \\
        A_{h,t} \arrow[r, "\dual"] & A_{h,t}^*.
    \end{tikzcd}
\]
Explicitly, the standard bases of $A$ and $A^*$ correspond as follows:
\begin{center}
    \input{tikzpictures/dual-correspondence}
\end{center}

Recall that $A_{h,t}^*$ admits a \textit{dual Frobenius algebra structure} with multiplication $\Delta^*$, unit $\epsilon^*$, comultiplication $m^*$, and counit $\iota^*$. 
One can see that behavior of the operations of $A_{h,t}^*$ with respect to the basis $\{\dual(1), \dual(X) \}$ is exactly that of $A_{h,t}$ with respect to the basis $\{1, X\}$ turned around. 
Thus $\dual$ is an isomorphism of Frobenius algebras. 
Furthermore, the following proposition states that the pair $(\hsigma_0, \hsigma_\dual)$ makes $(R_{h, t}, A_{h,t}^*)$ an involutive Frobenius extension, and that $\dual$ is an isomorphism of involutive Frobenius extensions.

\begin{prop}
\label{prop:hat-sigma-and-dual-ops}
    The involution $\hsigma_\dual$ on $A_{h,t}^*$ satisfies
    \begin{gather*}
        \Delta^* \circ (\hsigma_\dual \otimes \hsigma_\dual) = \hsigma_\dual \circ \Delta^*, \quad
        \epsilon^* \circ \hsigma_0 = \hsigma_\dual \circ \epsilon^*, \\
        m^* \circ \hsigma_\dual = (\hsigma_\dual \otimes \hsigma_\dual) \circ m^*,\quad 
        \iota^* \circ \hsigma_\dual = \hsigma_0 \circ \iota^*.
    \end{gather*}
\end{prop}

\begin{proof}
    For the first equation, consider the following cubical diagram, where $A:=A_{h,t}$ for brevity:
    \[
\begin{tikzcd} 
& A \otimes A \arrow[ld, "\hsigma \otimes \hsigma" description] \arrow[dd, "m" description, pos=.75] \arrow[rr, "\dual" description] & & A^* \otimes A^* \arrow[dd, "\Delta^*" description] \arrow[ld, "\hsigma_\dual \otimes \hsigma_\dual" description] \\
A \otimes A \arrow[dd, "m" description] \arrow[rr, "\dual" description, pos=.75] & & A^* \otimes A^* \arrow[dd, "\Delta^*" description, pos=.75] & \\ 
& A \arrow[ld, "\hsigma" description] \arrow[rr, "\dual" description, pos=.75] & & A^* \arrow[ld, "\hsigma_\dual" description] \\
A \arrow[rr, "\dual" description] & & A^* & 
\end{tikzcd}        
    \]
    The commutativity of the right face follows from the commutativity of the other faces. The proof for the other cases are similar. 
\end{proof}

For a link diagram $D$, let $\CKh_{h, t}(D)^*$ denote the algebraic dual of $\CKh_{h, t}(D)$. The involution $\hsigma_\dual$ is extended to an involution on $\CKh_{h, t}(D)^*$ by 
\[
    \hsigma_\dual \colon \CKh_{h, t}(D)^* \to \CKh_{h, t}(D)^*;\quad
    x \mapsto (\hsigma_\dual \otimes \cdots \otimes \hsigma_\dual)(x), \ \ \mathrm{for} \ \ x\in (A_{h,t}^*)^{\otimes k}\subset \CKh_{h, t}(D)^*.
\]

\begin{prop}
    For a link diagram $D$, let $D^*$ denote the mirror of $D$. There is a canonical chain isomorphism
    \[
        \dual\colon \CKh_{h, t}(D^*) \cong \CKh_{h, t}(D)^*
    \]
    which commutes with the respective involutions:
    \[
\begin{tikzcd}
\CKh_{h, t}(D^*) \arrow[d, "\hsigma"'] \arrow[r, "\sim"] & \CKh_{h, t}(D)^* \arrow[d, "\hsigma_\dual"] \\
\CKh_{h, t}(D^*) \arrow[r, "\sim"] & \CKh_{h, t}(D)^*.           
\end{tikzcd}
    \]
\end{prop}

\begin{proof}
    The cube of resolutions for $D^*$ can be obtained from that for $D$ by replacing each vertex $v$ with $\bar{v}$ where $\bar{v}_i = 1 - v_i$, which gives identical resolutions $D_v = D^*_{\bar{v}}$, and reversing each edge $e_{uv}\colon D_u \rightarrow D_v$ to $\bar{e}_{uv}\colon D^*_{\bar{u}} \leftarrow D^*_{\bar{v}}$. From this observation, one can see that the correspondence
    \[
        \dual\colon x_1 \otimes \cdots \otimes x_r \in A_{h,t}^{\otimes r} \mapsto \dual(x_1) \otimes \cdots \otimes \dual(x_r) \in (A_{h,t}^*)^{\otimes r}
    \]
    gives a chain isomorphism $\CKh_{h, t}(D^*) \cong \CKh_{h, t}(D)^*$. That $\dual$ commutes with the involutions is immediate from the definition of $\hsigma_\dual$. 
\end{proof}

\subsection{Various Frobenius extensions}
\label{subsec:frob-ext}

Various Frobenius extensions are considered in \cite{Khovanov:2004,Khovanov:2022}. These extensions can be endowed with involutive structures that extend the one defined in \Cref{subsec:hsigma} as follows. 

\begin{enumerate}
    \item The \textit{$U(1) \times U(1)$-equivariant theory} is given by the Frobenius extension $R_{\alpha} := \Z[\alpha_1,\alpha_2]$ and $A_{\alpha} := R_{\alpha}[X]/((X-\alpha_1)(X-\alpha_2))$ with $\deg \alpha_1 = \deg \alpha_2 = 2$. The inclusion $R_{h, t} \subset R_\alpha$ is given by $h = \alpha_1 + \alpha_2$, $t = -\alpha_1\alpha_2$. One can see that $\hsigma$ naturally extends over $R_{\alpha}$ and $A_{\alpha}$ with
    \[
        \hsigma(\alpha_1) = - \alpha_1,\quad 
        \hsigma(\alpha_2) = - \alpha_2
    \]
    and 
    \[
        \hsigma(X - \alpha_1) = X - \alpha_2,\quad 
        \hsigma(X - \alpha_2) = X - \alpha_1.
    \]
    There is an additional symmetry $\sigma_\alpha$ that transposes the roots $\alpha_1, \alpha_2$ and fixes $X$, 
    \[
        \sigma_\alpha(\alpha_1) = \alpha_2,\quad 
        \sigma_\alpha(\alpha_2) = \alpha_1,\quad
        \sigma_\alpha(X) = X. 
    \]
    


    \item The \textit{$U(1)$-equivariant theory} is given by $R_h := \Z[h]$ and $A_h := R_h[X]/(X^2 - hX)$ with $\deg(h) = 2$. There is an obvious mapping $(R_{h, t}, A_{h, t}) \to (R_h, A_h)$ by setting $t = 0$. This theory was originally introduced by Bar-Natan~\cite{BarNatan:2005} over $\mathbb{F}_2$, where $H$ is used instead of $h$.

    \item The \textit{$SU(2)$-equivariant theory}  is given by $R_t := \Z[t]$ and $A_t := R_t[X]/(X^2 - t)$ with $\deg(t) = 4$. Here, the gradings are $0 \pmod{4}$, so we define $\hsigma_{R_t} = \id_{R_t}$ and $\hsigma_{A_t} = \id_{A_t}$. It is also called Lee's theory~\cite{Lee:2005}, together with its version given by localizing $R_t$ to $\Z[t,t^{-1}]$ and likewise for $A_t$. One can also consider a rank $2$ extension $R_{\sqrt{t}} := \Z[\sqrt{t}]$ and $A_{\sqrt{t}} := R_{\sqrt{t}}[X]/(X^2 - t)$ of $(R_t, A_t)$, which will be revisited in \Cref{sec:rasmussen}. 

    \item The original (non-equivariant) theory of \cite{Khovanov:2000} (with $c = 0$) is given by $R_0 = \Z$ and $A_0 = R_0[X]/(X^2)$. The involutions are given by the identity maps. 
\end{enumerate}

The diagram of \Cref{fig:frob-exts} depicts the relation between these involutive Frobenius extensions, where arrows are involutive homomorphisms given by base changes indicated by the labels. 

\begin{figure}[t]
    \centering
\[
\begin{tikzcd}[row sep=4em, column sep=4em]
& {(R_\alpha, A_\alpha)} 
    \arrow[ldd, "\substack{\alpha_1 = 0,\\ \alpha_2 = h}" description, bend right]
    \arrow[rrdd, "\substack{\alpha_1 = -\sqrt{t},\\ \alpha_2 = \sqrt{t}}" description, bend left] & \\ 
& {(R_{h, t}, A_{h, t})} 
    \arrow[u, "\substack{h = \alpha_1 + \alpha_2,\\ t = -\alpha_1\alpha_2}" description, hook] 
    \arrow[ld, "t = 0" description] 
    \arrow[rd, "h = 0" description] & \\
{(R_h, A_h)} 
    \arrow[rd, "h = 0" description] 
& & {(R_t, A_t)}
    \arrow[r, hook]
    \arrow[ld, "t = 0" description]
& {(R_{\sqrt{t}}, A_{\sqrt{t}})} 
    \arrow[lld, "\sqrt{t} = 0" description, bend left = 20] \\
& {(R_0, A_0)}     
\end{tikzcd}
\]
    \caption{Involutive Frobenius extensions}
    \label{fig:frob-exts}
\end{figure}
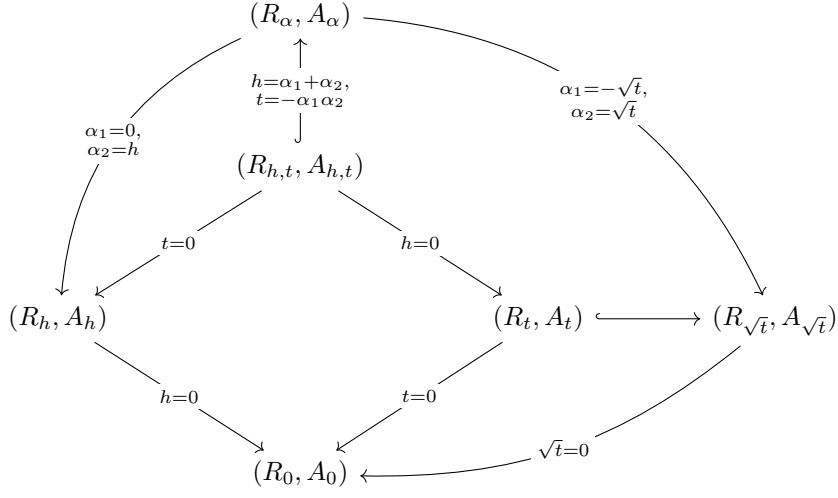
\section{Shumakovitch operation and reduced theories}

\subsection{Shumakovitch operation \texorpdfstring{$\nu$}{nu}}

In \cite{Shumakovitch:2014}, Shumakovitch introduced an operation $\nu$ on the $\F_2$-Khovanov homology $\Kh_0(-; \F_2)$, and proved that the unreduced $\F_2$-Khovanov homology $\Kh_0(-; \F_2)$ splits as the direct sum of two copies of the reduced $\F_2$-Khovanov homology $\widetilde{\Kh}_0(-; \F_2)$. In \cite{Wigderson:2016}, Wigderson extended the operation to the $\F_2$-Bar-Natan homology ($U(1)$-equivariant homology over $\F_2$) and proved that the unreduced homology splits into the direct sum of two copies of the reduced homology. Here, we briefly review the definition of $\nu$ and its extended version. 

The \textit{Shumakovitch operation} $\nu$ is defined as follows: for any $r \geq 1$ and any element $x = x_1 \otimes \cdots \otimes x_r \in A_0^{\otimes r} \otimes \F_2$ with $x_i \in \{1, X\}$, the element $\nu(x)$ is defined as the sum of all elements obtained by choosing one factor $x_i$ labeled $X$ and replacing it with $1$, 
\[
    \nu(x) := \sum_{x_i = X} x_1 \otimes \cdots \otimes 1 \otimes \cdots \otimes x_r.
\]
The operation $\nu$ can be given a visual description as follows:
\begin{center}
    \tikzset{every picture/.style={line width=0.75pt}} 

\begin{tikzpicture}[x=0.75pt,y=0.75pt,yscale=-1,xscale=1]

\draw  [color={rgb, 255:red, 0; green, 0; blue, 0 }  ,draw opacity=1 ] (113.09,25.17) .. controls (113.09,17.46) and (115.67,11.22) .. (118.85,11.22) .. controls (122.03,11.22) and (124.6,17.46) .. (124.6,25.17) .. controls (124.6,32.88) and (122.03,39.12) .. (118.85,39.12) .. controls (115.67,39.12) and (113.09,32.88) .. (113.09,25.17) -- cycle ;
\draw [color={rgb, 255:red, 0; green, 0; blue, 0 }  ,draw opacity=1 ]   (118.85,39.12) -- (178.11,39.12) ;
\draw [color={rgb, 255:red, 0; green, 0; blue, 0 }  ,draw opacity=1 ]   (118.85,11.22) -- (178.11,11.22) ;
\draw  [draw opacity=0][dash pattern={on 0.84pt off 2.51pt}] (175.82,38.91) .. controls (173.6,37.26) and (171.97,31.82) .. (171.97,25.37) .. controls (171.97,18.93) and (173.6,13.5) .. (175.81,11.84) -- (177.24,25.37) -- cycle ; \draw  [dash pattern={on 0.84pt off 2.51pt}] (175.82,38.91) .. controls (173.6,37.26) and (171.97,31.82) .. (171.97,25.37) .. controls (171.97,18.93) and (173.6,13.5) .. (175.81,11.84) ;  
\draw  [draw opacity=0] (176.67,11.4) .. controls (176.86,11.34) and (177.05,11.32) .. (177.24,11.32) .. controls (180.16,11.32) and (182.52,17.61) .. (182.52,25.37) .. controls (182.52,33.14) and (180.16,39.43) .. (177.24,39.43) .. controls (177.05,39.43) and (176.86,39.4) .. (176.67,39.35) -- (177.24,25.37) -- cycle ; \draw   (176.67,11.4) .. controls (176.86,11.34) and (177.05,11.32) .. (177.24,11.32) .. controls (180.16,11.32) and (182.52,17.61) .. (182.52,25.37) .. controls (182.52,33.14) and (180.16,39.43) .. (177.24,39.43) .. controls (177.05,39.43) and (176.86,39.4) .. (176.67,39.35) ;  

\draw  [color={rgb, 255:red, 0; green, 0; blue, 0 }  ,draw opacity=1 ] (113.09,128.17) .. controls (113.09,120.46) and (115.67,114.22) .. (118.85,114.22) .. controls (122.03,114.22) and (124.6,120.46) .. (124.6,128.17) .. controls (124.6,135.88) and (122.03,142.12) .. (118.85,142.12) .. controls (115.67,142.12) and (113.09,135.88) .. (113.09,128.17) -- cycle ;
\draw [color={rgb, 255:red, 0; green, 0; blue, 0 }  ,draw opacity=1 ]   (118.85,142.12) -- (178.11,142.12) ;
\draw [color={rgb, 255:red, 0; green, 0; blue, 0 }  ,draw opacity=1 ]   (118.85,114.22) -- (178.11,114.22) ;
\draw  [draw opacity=0][dash pattern={on 0.84pt off 2.51pt}] (175.82,141.91) .. controls (173.6,140.26) and (171.97,134.82) .. (171.97,128.37) .. controls (171.97,121.93) and (173.6,116.5) .. (175.81,114.84) -- (177.24,128.37) -- cycle ; \draw  [dash pattern={on 0.84pt off 2.51pt}] (175.82,141.91) .. controls (173.6,140.26) and (171.97,134.82) .. (171.97,128.37) .. controls (171.97,121.93) and (173.6,116.5) .. (175.81,114.84) ;  
\draw  [draw opacity=0] (176.67,114.4) .. controls (176.86,114.34) and (177.05,114.32) .. (177.24,114.32) .. controls (180.16,114.32) and (182.52,120.61) .. (182.52,128.37) .. controls (182.52,136.14) and (180.16,142.43) .. (177.24,142.43) .. controls (177.05,142.43) and (176.86,142.4) .. (176.67,142.35) -- (177.24,128.37) -- cycle ; \draw   (176.67,114.4) .. controls (176.86,114.34) and (177.05,114.32) .. (177.24,114.32) .. controls (180.16,114.32) and (182.52,120.61) .. (182.52,128.37) .. controls (182.52,136.14) and (180.16,142.43) .. (177.24,142.43) .. controls (177.05,142.43) and (176.86,142.4) .. (176.67,142.35) ;  

\draw  [draw opacity=0][dash pattern={on 0.84pt off 2.51pt}] (175.7,91.61) .. controls (173.49,89.96) and (171.86,84.52) .. (171.86,78.07) .. controls (171.86,71.63) and (173.48,66.2) .. (175.7,64.54) -- (177.13,78.07) -- cycle ; \draw  [dash pattern={on 0.84pt off 2.51pt}] (175.7,91.61) .. controls (173.49,89.96) and (171.86,84.52) .. (171.86,78.07) .. controls (171.86,71.63) and (173.48,66.2) .. (175.7,64.54) ;  
\draw  [draw opacity=0] (176.55,64.1) .. controls (176.74,64.04) and (176.93,64.02) .. (177.13,64.02) .. controls (180.04,64.02) and (182.4,70.31) .. (182.4,78.07) .. controls (182.4,85.84) and (180.04,92.13) .. (177.13,92.13) .. controls (176.93,92.13) and (176.74,92.11) .. (176.55,92.05) -- (177.13,78.07) -- cycle ; \draw   (176.55,64.1) .. controls (176.74,64.04) and (176.93,64.02) .. (177.13,64.02) .. controls (180.04,64.02) and (182.4,70.31) .. (182.4,78.07) .. controls (182.4,85.84) and (180.04,92.13) .. (177.13,92.13) .. controls (176.93,92.13) and (176.74,92.11) .. (176.55,92.05) ;  
\draw  [draw opacity=0] (177.71,64.02) .. controls (177.51,64.02) and (177.32,64.02) .. (177.13,64.02) .. controls (168,64.02) and (160.61,70.31) .. (160.61,78.07) .. controls (160.61,85.84) and (168,92.13) .. (177.13,92.13) .. controls (177.32,92.13) and (177.51,92.13) .. (177.71,92.12) -- (177.13,78.07) -- cycle ; \draw   (177.71,64.02) .. controls (177.51,64.02) and (177.32,64.02) .. (177.13,64.02) .. controls (168,64.02) and (160.61,70.31) .. (160.61,78.07) .. controls (160.61,85.84) and (168,92.13) .. (177.13,92.13) .. controls (177.32,92.13) and (177.51,92.13) .. (177.71,92.12) ;  

\draw  [color={rgb, 255:red, 0; green, 0; blue, 0 }  ,draw opacity=1 ] (113,77.69) .. controls (113,70.42) and (115.43,64.53) .. (118.43,64.53) .. controls (121.42,64.53) and (123.85,70.42) .. (123.85,77.69) .. controls (123.85,84.95) and (121.42,90.84) .. (118.43,90.84) .. controls (115.43,90.84) and (113,84.95) .. (113,77.69) -- cycle ;
\draw  [draw opacity=0] (117.88,64.44) .. controls (118.06,64.44) and (118.24,64.43) .. (118.43,64.43) .. controls (127.03,64.43) and (134,70.37) .. (134,77.69) .. controls (134,85) and (127.03,90.94) .. (118.43,90.94) .. controls (118.24,90.94) and (118.06,90.93) .. (117.88,90.93) -- (118.43,77.69) -- cycle ; \draw   (117.88,64.44) .. controls (118.06,64.44) and (118.24,64.43) .. (118.43,64.43) .. controls (127.03,64.43) and (134,70.37) .. (134,77.69) .. controls (134,85) and (127.03,90.94) .. (118.43,90.94) .. controls (118.24,90.94) and (118.06,90.93) .. (117.88,90.93) ;  

\draw   (89,10) .. controls (84.33,10) and (82,12.33) .. (82,17) -- (82,65.99) .. controls (82,72.66) and (79.67,75.99) .. (75,75.99) .. controls (79.67,75.99) and (82,79.32) .. (82,85.99)(82,82.99) -- (82,134.98) .. controls (82,139.65) and (84.33,141.98) .. (89,141.98) ;
\draw   (194,138.98) .. controls (198.67,138.98) and (201,136.65) .. (201,131.98) -- (201,83.49) .. controls (201,76.82) and (203.33,73.49) .. (208,73.49) .. controls (203.33,73.49) and (201,70.16) .. (201,63.49)(201,66.49) -- (201,14.98) .. controls (201,10.31) and (198.67,7.98) .. (194,7.98) ;

\draw (6,62.4) node [anchor=north west][inner sep=0.75pt]    {$\displaystyle \nu \ =\ \sum_{i=1}^r$};
\draw (141,43.4) node [anchor=north west][inner sep=0.75pt]    {$\vdots $};
\draw (141,91.4) node [anchor=north west][inner sep=0.75pt]    {$\vdots $};
\draw (111.09,25.17) node [anchor=east] [inner sep=0.75pt]    {$1$};
\draw (111,77.69) node [anchor=east] [inner sep=0.75pt]    {$i$};
\draw (111.09,128.17) node [anchor=east] [inner sep=0.75pt]    {$r$};

\end{tikzpicture}
\end{center}
One can prove that $\nu$ commutes with the operations of $A_0^{\otimes r} \otimes \F_2$, so that it is an Frobenius algebra endomorphism. This will be reproved later in a more generalized setting. 

Furthermore, the operation $\nu$ can be extended as a Frobenius algebra endomorphism to the $U(1)$-equivariant setting as follows: for each $1 \leq k \leq r$ and for any element $x = x_1 \otimes \cdots \otimes x_r \in A_h^{\otimes r} \otimes \F_2$ with $x_i \in \{1, X\}$, $\nu_k(x)$ is defined as the sum of all elements obtained by choosing $k$ factors labeled $X$ and replacing them with $1$. The map $\nu_k$ has degree $-2k$, and can be given a similar cobordism description as above by a sum of $\binom{r}{k}$ cobordisms, each with $k$ cups and $k$ opposite caps (instead of a single cup and cap in position $i$ in the above formula for $\nu$). Define a degree $-2$ endomorphism $\bar{\nu}$ on $A_h^{\otimes r} \otimes \F_2$ by
\[
    \bar{\nu} := \sum^r_{k = 1} h^{k - 1} \nu_k.
\]
By definition, setting $h = 0$ recovers the original Shumakovitch operation $\nu$. For convenience, we let $\nu_0 = \id$ and $\nu_k = 0$ for $k > r$. The operation $\bar{\nu}$ and the involution $\sigma = \sigma^{\otimes r}$ are related as follows:
\begin{prop}
\label{prop:nu-and-sigma}
\[
    \id + h \bar{\nu} = \sigma.
\]
\end{prop}

\begin{proof}
    The proof proceeds by induction on  $r$. When $r = 1$, we have
    \begin{gather*}
        1 + h \bar{\nu}(1) = 1 = \sigma(1),\\
        X + h \bar{\nu}(X) = X + h = \sigma(X).
    \end{gather*}
    Next, suppose $r > 1$ and the result holds for $r - 1$. We may write 
    \begin{align*}
        \bar{\nu}
        &= \sum_k h^{k - 1} (\bar{\nu}^{(1)}_1 \otimes \bar{\nu}^{(r-1)}_{k-1} + \id^{(1)} \otimes \bar{\nu}^{(r-1)}_k) \\
        &= \bar{\nu}^{(1)}_1 \otimes \id^{(r-1)} + (\id^{(1)} + h\bar{\nu}^{(1)}_1) \otimes \bar{\nu}^{(r-1)}.
    \end{align*}
    Here, each superscript $(k)$ indicates that it is an endomorphism on $A^{\otimes k}$. With this equation, we have 
    \begin{align*}
        \id + h \bar{\nu}
        &= \id^{(1)} \otimes \id^{(r-1)} + h \left(\bar{\nu}^{(1)}_1 \otimes \id^{(r-1)} + (\id^{(1)} + h\bar{\nu}^{(1)}_1) \otimes \bar{\nu}^{(r-1)} \right) \\ 
        &= (\id^{(1)} + h \bar{\nu}^{(1)}) \otimes (\id^{(r - 1)} + h \bar{\nu}^{(r - 1)})
    \end{align*}
    and the proof is immediate. 
\end{proof}

From \Cref{prop:nu-and-sigma}, the operation $\bar{\nu}$ can be alternatively defined as
\[
    \bar{\nu} = \frac{\id + \sigma}{h}.
\]
This description allows us to extend $\bar{\nu}$ to the signed setting, using the signed involution $\hsigma$. 

\subsection{Signed Shumakovitch operation \texorpdfstring{$\hnu$}{nu-hat}}
\label{subsec:hnu}

Here, we consider the $U(2)$-equivariant Frobenius extension $(R_{h, t}, A_{h, t})$. 

\begin{lem}
\label{lem:id-minus-sigma}
    For any $r \geq 0$, the endomorphism $\id - \hsigma$ on $A_{h, t}^{\otimes r}$ is divisible by $h$. 
\end{lem}

\begin{proof}
    When $r = 0$, we have 
    \[
        1 - \hsigma(1) = 0,\quad 
        h - \hsigma(h) = 2h,\quad 
        t - \hsigma(t) = 0
    \]
    and when $r = 1$,
    \[
        X - \hsigma(X) = h,\quad 
        hX - \hsigma(hX) = h(2X - h).
    \]
    For $r > 1$, we have 
    \begin{align*}
        (\id - \hsigma)(x_1 \otimes x_2 \otimes \cdots \otimes x_r) &= (\id - \hsigma)(x_1) \otimes x_2 \otimes \cdots \otimes x_r \\
        &+ \hsigma(x_1) \otimes (\id - \hsigma)(x_2) \otimes \cdots \otimes x_r \\
        &+ \cdots \\
        &+ \hsigma(x_1) \otimes \hsigma(x_2) \otimes \cdots \otimes (\id - \hsigma)(x_r)        
    \end{align*}
    so the claim follows by induction. 
\end{proof}

The above proof can be given a diagrammatic description. For example, when $r = 1$, 
\begin{center}
    \input{tikzpictures/id-minus-sigma}
\end{center}

From \Cref{lem:id-minus-sigma}, we may define the \textit{signed Shumakovitch operation} $\hnu$ as a ($\Z[h^2, t]$-module) endomorphism of $A_{h, t}^{\otimes r}$ by 
\[
    \hnu := \frac{\id - \hsigma}{h}. 
\]
Obviously, $\hnu$ has degree $-2$. Explicitly, we have 
\begin{gather*}
    \hnu(1) = 0,\quad \hnu(h) = 2,\quad \hnu(t) = 0,\quad \hnu(X) = 1,\quad \hnu(Y) = -1, \\
    \hnu(hX) = \hnu(hY) = 2X - h = U.
\end{gather*}

Again, observe that $\hnu$ is not an $R_{h,t}$-module involution, but an $R'_{h, t}$-module involution, where $R'_{h, t} = \Z[h^2, t]$. Also note that we need $h \neq 0$ to define $\hnu$. By setting $t = 0$ and tensoring the ground ring with $\F_2$, it is immediate from \Cref{prop:nu-and-sigma} that $\hnu$ recovers the extended Shumakovitch operation $\bar{\nu}$. 

\begin{remark}
\label{rem:hnu-over-U1xU1}
    The operation $\hnu$ does not extend to the $U(1) \times U(1)$-equivariant theory, since 
    \[
        (\id - \hsigma)(\alpha_i) = 2 \alpha_i
    \]
    is not divisible by $h = \alpha_1 + \alpha_2$. This problem will be revisited in \Cref{subsec:U1xU1-equiv}. 
\end{remark}

The following propositions generalize the properties of $\nu$ proved in \cite[Section 3]{Shumakovitch:2014}.

\begin{prop}
\label{prop:hnu-formula} $ $
    \begin{enumerate}
        \item $\hnu(x \otimes y) = \hnu(x) \otimes y + \hsigma(x) \otimes \hnu(y).$
        \item $\hsigma \, \hnu = - \hnu\, \hsigma = \hnu.$
        \item $\hnu^2 = 0$.
    \end{enumerate}
\end{prop}

\begin{proof}
    (1), (2) are immediate from the definition of $\hnu$, and (3) follows from (2).
\end{proof}

From \Cref{prop:hnu-formula} (1), one can easily compute, for instance,
\[
    \hnu(X \otimes X) = 1 \otimes X + Y \otimes 1 = 1 \otimes X + X \otimes 1 - h(1 \otimes 1)
\]
and 
\[
    \hnu(X \otimes Y) = 1 \otimes Y - Y \otimes 1 = 1 \otimes X - X \otimes 1.
\]

\begin{prop}
\label{prop:hnu-and-A}
    $\hnu$ commutes with the Frobenius algebra operations on $A_{h,t}$.
\end{prop}

\begin{proof}
    Since $\id, m, \iota, \Delta, \epsilon$ are $R_{h,t}$-module homomorphisms and the factor $\frac{1}{h}$ can be treated as a scalar, the result is immediate from \Cref{prop:hsigma-ops}. 
\end{proof}

Let $\Xbar$ denote the degree $2$ endomorphism on $A_{h, t}^{\otimes r}$ for $r \geq 1$ defined by 
\[
    \Xbar(x_1 \otimes x_2 \otimes \cdots \otimes x_r) := (X \cdot x_1) \otimes x_2 \otimes \cdots \otimes x_r. 
\]
Similarly define an endomorphism $\Ybar$.

\begin{prop}
\label{prop:hnu-acyclic}
\begin{align*}
    \hnu \circ \Xbar - \Ybar \circ \hnu &= \id,\\
    -\hnu \circ \Ybar + \Xbar \circ \hnu &= \id.
\end{align*}
\end{prop}

\begin{proof}
    Put $x = x_1 \otimes y \in A_{h,t} \otimes A_{h,t}^{\otimes(r - 1)}$. We compute
    \begin{align*}
        \hnu (\Xbar(x))
        &= \hnu((Xx_1) \otimes y) \\
        &= \hnu(Xx_1) \otimes y + \hsigma(Xx_1) \otimes \hnu(y) \\ 
        &= (x_1 + Y \hnu(x_1)) \otimes y + Y\hsigma(x_1) \otimes \hnu(y) \\
        &= x + \Ybar(\hnu(x)).
    \end{align*}
    The second equation can be proved similarly. 
\end{proof}

Now, let $D$ be a link diagram. Regarding $\CKh_{h, t}(D)$ as a (infinitely generated) bigraded $\Z$-chain complex and $\Kh_{h, t}(D)$ as a bigraded $\Z$-module, \Cref{prop:hnu-formula,prop:hnu-and-A} imply that $\hnu$ induces an endomorphism on $\Kh_{h, t}(D)$ that squares to $0$. If $D$ is non-empty, we may fix a basepoint $p$ on $D$ so that endomorphisms $\Xbar, \Ybar$ are well defined on $\CKh_{h, t}(D)$ and on $\Kh_{h, t}(D)$. For each homological grading $i$, consider the sequence
\[
\begin{tikzcd}
\cdots \arrow[r, "-\Ybar"', shift right] & {\Kh^{i,j-2}(D)} \arrow[r, "\Xbar"', shift right] \arrow[l, "\hnu"', shift right] & {\Kh^{i,j}(D)} \arrow[l, "\hnu"', shift right] \arrow[r, "-\Ybar"', shift right] & {\Kh^{i,j+2}(D)} \arrow[l, "\hnu"', shift right] \arrow[r, "\Xbar"', shift right] & \cdots. \arrow[l, "\hnu"', shift right]
\end{tikzcd}
\]
\Cref{prop:hnu-acyclic} implies that $\id$ is null-homotopic with respect to the differential $\hnu$. Thus we obtain a generalization of \cite[Theorem 3.2.A]{Shumakovitch:2014}.

\begin{prop}
    If $D \neq \emptyset$, the complex $(\Kh_h(D), \hnu)$ is acyclic.
\end{prop}

\subsection{Reduced homology and splitting}
\label{subsec:reduced}

Here, we consider the $U(1)$-equivariant theory over the ground ring $R_h=\Z[h]$. Let $D$ be a pointed link diagram with marked point $p$. Here, for notational simplicity, we write $C = \CKh_h(D)$, and let $C_X$ denote the subcomplex of $C$ generated by enhanced states of the form
\[
    x = \underline{X} \otimes x_2 \otimes \cdots \otimes x_r,
\]
where the underline indicates the factor corresponding to the circle containing $p$. The subcomplex $C_Y$ is defined similarly. The two complexes are isomorphic, and either one may well be called the \textit{reduced $U(1)$-equivariant Khovanov complex} of $D$, denoted $\rCKh_h(D)$. (Later, in \Cref{sec:rasmussen}, we choose the one that contains the \textit{Lee cycle} $\ca(D)$ of $D$). Its homology is called the \textit{reduced $U(1)$-equivariant Khovanov homology} and denoted $\rKh_h(D)$. It is conventional to shift the quantum grading of the reduced complex by $-1$ so that we have 
\[
    \rCKh_h(\bigcirc) = \rKh_h(\bigcirc) \cong R_h. 
\]

\begin{figure}[t]
    \centering
    \begin{subfigure}{0.45\linewidth}
\centering
\begin{tabular}{r|lll}
$7$ & $\F_2[h]/(h)$ & $.$ & $.$ \\
$5$ & $\F_2[h]/(h)$ & $.$ & $.$ \\
$3$ & $.$ & $.$ & $\F_2[h]$ \\
$1$ & $.$ & $.$ & $\F_2[h]$ \\
\hline
$ $ & $-2$ & $-1$ & $0$ \\
\end{tabular}
\caption{$\Kh_h(3_1; \F_2)$}
    \end{subfigure}
    \begin{subfigure}{0.45\linewidth}
\centering
\begin{tabular}{r|lll}
$6$ & $\F_2[h]/(h)$ & $.$ & $.$ \\
$4$ & $.$ & $.$ & $.$ \\
$2$ & $.$ & $.$ & $\F_2[h]$ \\
\hline
$ $ & $-2$ & $-1$ & $0$ \\
\end{tabular}
\caption{$\rKh_h(3_1; \F_2)$}        
    \end{subfigure}
    \par\vspace{2em}
    \begin{subfigure}{0.45\linewidth}
\centering
\begin{tabular}{r|lll}
$5$ & $\Q[h]/(h^2)$ & $.$ & $.$ \\
$3$ & $.$ & $.$ & $\Q[h]$ \\
$1$ & $.$ & $.$ & $\Q[h]$ \\
\hline
$ $ & $-2$ & $-1$ & $0$ \\
\end{tabular}
\caption{$\Kh_h(3_1; \Q)$}
    \end{subfigure}
    \begin{subfigure}{0.45\linewidth}
\centering
\begin{tabular}{r|lll}
$6$ & $\Q[h]/(h)$ & $.$ & $.$ \\
$4$ & $.$ & $.$ & $.$ \\
$2$ & $.$ & $.$ & $\Q[h]$ \\
\hline
$ $ & $-2$ & $-1$ & $0$ \\
\end{tabular}
\caption{$\rKh_h(3_1; \Q)$}    
    \end{subfigure}
    \caption{The unreduced and reduced $U(1)$-equivariant homology for the trefoil $3_1$ over $\F_2$ and $\Q$.}
    \label{tab:compare-CKh-for-F2-and-Q}
\end{figure}

\Cref{tab:compare-CKh-for-F2-and-Q} shows the unreduced and reduced $U(1)$-Khovanov homologies for the trefoil over $\F_2$ and $\Q$.\footnote{
    The computation is performed using the program \verb|YUI| developed by the second author~\cite{Sano:YUI}.    
} Over $\F_2$, one can see that the unreduced homology splits into two copies of the reduced homology, as generally proved by Wigderson on the complex level \cite{Wigderson:2016}. On the other hand, over $\Q$, there is a single torsion summand $\Q[h]/(h^2)$ in the unreduced homology, which is not isomorphic to the direct sum of two copies of $\Q[h]/(h)$ in the reduced homology. Thus, the splitting property does not hold over $\Q[h]$. However, if we consider over the subring $\Q[h^2]$, then we have 
\[
    \Q[h]/(h^2) \cong \Q[h]/(h) \oplus q^2 \Q[h]/(h)
    \ \ \text{(over $\Q[h^2]$)}.
\]
The following theorem states that this holds in general. 

\begin{thm}
\label{thm:split}
    Over the subring $R'_h = \Z[h^2]$ of $R_h$, the unreduced $U(1)$-equivariant Khovanov complex $\CKh_h(D)$ splits into the direct sum of two copies of the reduced complex $\rCKh_h(D)$. 
\end{thm}

\Cref{thm:split} follows from the following proposition.

\begin{prop}
\label{prop:red-ses}
    The two rows of the following diagram are short exact sequences of $R_h$-complexes, and the involution $\hsigma$ gives an $R'_h$-isomorphism between them. Moreover, for each row, the restriction of the endomorphism $\hnu$ gives a splitting as $R'_h$-chain complexes.
    \[
\begin{tikzcd}[row sep=3em, column sep=4em]
0 \arrow[r] & C_X \arrow[r, hook] \arrow[d, "\hsigma"] & C \arrow[r, "\Ybar"] \arrow[d, "\hsigma"] & C_Y \arrow[r] \arrow[d, "\hsigma"] \arrow[l, "-\hnu", dashed, bend left] & 0 \\
0 \arrow[r] & C_Y \arrow[r, hook] & C \arrow[r, "\Xbar"] & C_X \arrow[r] \arrow[l, "\hnu", dashed, bend left] & 0.
\end{tikzcd}
    \]
\end{prop}

\begin{proof}
    The first two statements are straightforward. From \Cref{prop:hnu-acyclic}, we have 
    \[
        \hnu \circ \Xbar - \Ybar \circ \hnu = \id
    \]
    on $C$, but $\Xbar$ restricts to $0$ on $C_Y$, proving that $-\hnu$ gives a section of $\Ybar$. 
\end{proof}

\begin{remark}
    In characteristic $2$, the endomorphism appearing in \Cref{prop:red-ses} are already endomorphisms over the ground ring $\F_2[h]$, so a similar argument shows that the $\F_2$-Bar-Natan homology splits (\cite[Theorem 4]{Wigderson:2016}), and further setting $h = 0$ shows that the $\F_2$-Khovanov homology splits (\cite[Corollary 3.2.C]{Shumakovitch:2014}).
\end{remark}

\begin{remark}
    The explicit splitting of \Cref{thm:split} and that of \cite[Theorem 4]{Wigderson:2016} are related as follows. Here, we work in characteristic 2. First, as an $\F_2[h]$-module, we identify the quotient complex $C / C_X$ with the submodule $C_1$ of $C$, which is generated by enhanced states of the form
    \[
        x = \underline{1} \otimes x_2 \otimes \cdots \otimes x_r.
    \]
    We have $C = C_1 \oplus C_X$ as modules, and the differential $d$ of $C$ can be written as 
    \[
        d = \begin{pmatrix}
            d_1 & \\
            f & d_X
        \end{pmatrix}
    \]
    where $d_1, d_X$ are differentials of $C_1, C_X$ respectively, and $f$ is a chain map 
    \[
        f\colon C_1 \to C_X
    \]
    so that $C$ is the mapping cone of $f$. With $C_Y \cong C / C_X = C_1$, the short exact sequence of \Cref{prop:red-ses} can be rewritten as 
    \[
\begin{tikzcd}
0 \arrow[r] & C_X \arrow[r, hook] & C \arrow[r] & C_1 \arrow[r] \arrow[l, "s", dashed, bend left] & 0
\end{tikzcd}
    \]
    where the section $s$ maps 
    \[
        s(\underline{1} \otimes y) = 
        \hnu(\underline{Y} \otimes y) = \underline{1} \otimes y + \underline{X} \otimes \hnu(y). 
    \]
    Let $K$ denote the map 
    \[
        K\colon C_1 \to C_X; \quad 
        \underline{1} \otimes y \mapsto \underline{X} \otimes \hnu(y). 
    \]
    Then the isomorphism $C_1 \oplus C_X \to C$ is given by 
    \[
        \begin{pmatrix} \id \\ K & \id \end{pmatrix}.
    \]
    This $K$ is exactly the null-homotopy of $f$ constructed in \cite{Wigderson:2016}. Indeed, the equation 
    \[
        \begin{pmatrix}
            d_1 \\ 
            f & d_X
        \end{pmatrix}
        \begin{pmatrix} 
            \id \\ 
            K & \id 
        \end{pmatrix}
        = 
        \begin{pmatrix}
            \id \\ 
            K & \id
        \end{pmatrix}
        \begin{pmatrix}
            d_1 & \\
            & d_X
        \end{pmatrix}
    \]
    gives 
    \[
        f = d_X K + K d_1.
    \]
    Furthermore, Wigderson gives an explicit isomorphism
    \[
        \Xbar + hK \colon C_1 \to C_X
    \]
    which can be rewritten from \Cref{prop:nu-and-sigma} as 
    \[
        \Xbar + hK = \Xbar(\id + h \hnu) |_{C_1} = \Xbar \sigma  |_{C_1}.
    \]
    This is an alternative description of the isomorphism
    \[
\begin{tikzcd}
C_1 \arrow[r, "\sim"] & C_Y \arrow[r, "\sigma"] & C_X.
\end{tikzcd}        
    \]
\end{remark}

\begin{remark}
    \Cref{tab:compare-CKh-for-F2-and-Q} shows that, over $\Q[h^2]$, the torsion part $\Q[h]/(h^2)$ in the unreduced indeed splits as two copies of $\Q[h]/(h)$ in the reduced homology. This pattern arises very often in other examples, but is not always the case. A counterexample is given by the 38-crossing knot $K$ found by Manolescu and Marengon in \cite[Figure 1]{Manolescu-Marengon:2020}, for which the \textit{Knight Move Conjecture} \cite[Conjecture 1]{Bar-Natan:2002} does not hold. The unreduced homology of this knot $K$, computed over $\Q[h]$, has 
    \[
        q^{-9}\Q[h]/(h^4) \oplus q^{-9}\Q[h]/(h^2) \oplus q^{-7}\Q[h]/(h^2)
    \]
    in homological grading $2$, whereas the reduced homology has 
    \[
        q^{-8} \Q[h]/(h^3) \oplus q^{-8} \Q[h]/(h)
    \]
    in the same homological grading. 
    Generators of the unreduced and reduced homologies correspond as follows.
    
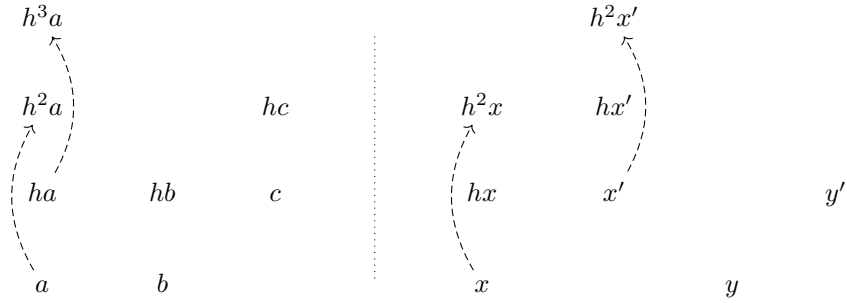
\[
\begin{tikzcd}
h^3a & & & {} \arrow[ddd, no head, dotted] & & h^2x' & & \\
h^2a & & hc & & h^2x & hx' & & \\
ha \arrow[uu, dashed, bend right] & hb & c & & hx & x' \arrow[uu, dashed, bend right] & & y' \\
a \arrow[uu, dashed, bend left] & b & & {} & x \arrow[uu, dashed, bend left] & & y &   
\end{tikzcd}
    \]
    Here, $a, b, c$ and $x, y, z$ are the corresponding generators in the unreduced and the reduced homologies, and $x', y', z'$ are copies of $x, y, z$. Dashed arrows indicate the multiplication by $h^2$. 
\end{remark}

\subsection{Splitting of \texorpdfstring{$U(1) \times U(1)$}{U(1) x U(1)}-equivariant theory}
\label{subsec:U1xU1-equiv}

As mentioned in \Cref{rem:hnu-over-U1xU1}, the operation $\hnu$ does not extend over the $U(1) \times U(1)$-equivariant theory $(R_\alpha, A_\alpha)$, if we regard $(R_\alpha, A_\alpha)$ as an involutive Frobenius extension via $\hsigma$. 
Nonetheless, recall that there is an additional symmetry $\sigma_\alpha$ of $(R_\alpha, A_\alpha)$ that transposes the roots $\alpha_1, \alpha_2$ and fixes $X$.
This involution $\sigma_\alpha$ on $(R_\alpha, A_\alpha)$ is an endomorphism over $\Z$, and can be regarded as an extension of $\hsigma$ on $(R_h, A_h)$ under the following inclusion
\[
    s \colon (R_h, A_h) \hookrightarrow (R_\alpha, A_\alpha);\quad 
    h \mapsto \alpha_2 - \alpha_1,\quad 
    X \mapsto X - \alpha_1. 
\]
Indeed, we have 
\[
\begin{tikzcd}
h \arrow[r, "s", maps to] \arrow[d, "\hsigma", maps to] & \alpha_2 - \alpha_1 \arrow[d, "\sigma_\alpha", maps to] & X \arrow[r, "s", maps to] \arrow[d, "\hsigma", maps to] & X - \alpha_1 \arrow[d, "\sigma_\alpha", maps to] \\
-h \arrow[r, "s", maps to] & \alpha_1 - \alpha_2 & X - h \arrow[r, "s", maps to] & X - \alpha_2
\end{tikzcd}
\]
Note that $s$ is a section of the projection
\[
    (R_\alpha, A_\alpha) \twoheadrightarrow (R_h, A_h);\quad 
    \alpha_1 \mapsto 0,\quad 
    \alpha_2 \mapsto h.
\]
This projection breaks the symmetry between $\alpha_1$ and $\alpha_2$. 
Similarly, define an involution $\sigma_{\sqrt{t}}$ on $(R_{\sqrt{t}}, A_{\sqrt{t}})$ by
\[
    \sigma_{\sqrt{t}}(\sqrt{t}) = -\sqrt{t}.
\]
These involutions fit into the following commutative diagram of involutive Frobenius extensions:
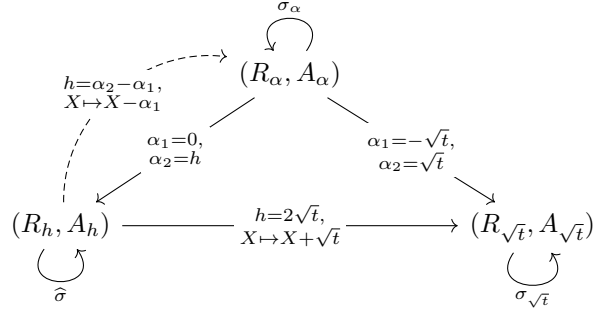
\[
\begin{tikzcd}[row sep=4em, column sep=4em]
& {(R_\alpha, A_\alpha)} \arrow[ld, "{\substack{\alpha_1 = 0,\\ \alpha_2 = h}}" description] \arrow[rd, "{\substack{\alpha_1 = -\sqrt{t},\\ \alpha_2 = \sqrt{t}}}" description] \arrow["\sigma_\alpha"', loop, distance=2em, in=125, out=55] & \\
{(R_h, A_h)} \arrow[ru, "{\substack{h = \alpha_2 - \alpha_1,\\X \mapsto X - \alpha_1}}" description, dashed, bend left=49] \arrow[rr, "{\substack{h = 2\sqrt{t},\\ X \mapsto X + \sqrt{t}}}" description] \arrow["\hsigma"', loop, distance=2em, in=305, out=235] & & {(R_{\sqrt{t}}, A_{\sqrt{t}})} \arrow["\sigma_{\sqrt{t}}"', loop, distance=2em, in=305, out=235]
\end{tikzcd} 
\]

Note that $\sigma_\alpha$ is incompatible with the involution $\hsigma$ on $(R_{h,t}, A_{h,t})$ under the inclusion 
\[
    (R_{h,t}, A_{h,t}) \hookrightarrow (R_\alpha, A_\alpha); \quad 
    h \mapsto \alpha_1 + \alpha_2, \quad 
    t \mapsto -\alpha_1\alpha_2,
\]
since we have $\hsigma(h) = -h$, but $\sigma_\alpha$ fixes $\alpha_1 + \alpha_2$. Instead, if we regard $(R_{h,t}, A_{h,t})$ as an involutive Frobenius extension with the trivial involution, which makes the above inclusion into an involutive homomorphism. 

Now, with $\sigma_\alpha$, we can define an operation $\nu_\alpha$ on $(R_\alpha, A_\alpha)$ that extends $\hnu$ on $(R_h, A_h)$. Indeed, 

\begin{prop}
    $\id - \sigma_\alpha$ is divisible by $c = \alpha_2 - \alpha_1$. 
\end{prop}

Thus, we may define an operation on $A^{\otimes r}_\alpha$ as
\[
    \nu_\alpha := \frac{\id - \sigma_\alpha}{\alpha_2 - \alpha_1}.
\]
In particular, we have 
\begin{gather*}
    \nu_\alpha(1) = 0,\quad 
    \nu_\alpha(\alpha_1) = -1,\quad 
    \nu_\alpha(\alpha_2) = 1, \quad 
    \nu_\alpha(X) = 0.
\end{gather*}
With $X_i := X - \alpha_i$ ($i = 1, 2$), we have 
\[
    \sigma_\alpha(X_1) = X_2,\quad 
    \sigma_\alpha(X_2) = X_1,\quad 
    \nu_\alpha(X_1) = -1,\quad 
    \nu_\alpha(X_2) = 1.
\]
Let $\Xbar_i$ ($i = 1, 2$) denote the endomorphism on $A_\alpha^{\otimes r}$ for $r \geq 1$ defined by 
\[
    \bar{X_i}(x_1 \otimes x_2 \otimes \cdots \otimes x_r) := (X_i \cdot x_1) \otimes x_2 \otimes \cdots \otimes x_r. 
\]
The following propositions are analogous to those proved in \Cref{subsec:hnu}, and are easy to verify. 

\begin{prop} $ $
    \begin{enumerate}
        \item $\nu_\alpha(x \otimes y) = \nu_\alpha(x) \otimes y + \sigma_\alpha(x) \otimes \nu_\alpha(y).$
        \item $\sigma_\alpha \, \nu_\alpha = - \nu_\alpha\, \sigma_\alpha = \nu_\alpha.$
        \item $\nu_\alpha^2 = 0$.
    \end{enumerate}
\end{prop}

\begin{prop}
    $\nu_\alpha$ commutes with the Frobenius algebra operations on $A_\alpha$.
\end{prop}

\begin{prop}
\label{prop:nu-alpha-acyclic}
    $\nu_\alpha \circ \bar{X_1} - \bar{X_2} \circ \nu_\alpha 
    = -\nu_\alpha \circ \bar{X_2} + \bar{X_1} \circ \nu_\alpha 
    = \id$.
\end{prop}

\begin{prop}
    If $D \neq \emptyset$, the complex $(\Kh_\alpha(D), \hnu)$ is acyclic.
\end{prop}

For a pointed link diagram $D$, we may define the reduced complex $\CKh_\alpha(D)$ as either one of the subcomplexes $C_i(D)$ ($i = 1, 2$) of the unreduced complex $C = C_\alpha(D)$ in the obvious way. From \Cref{prop:nu-alpha-acyclic}, one sees that a proposition analogous to \Cref{prop:red-ses} holds, and we obtain the following. 

\begin{thm}
\label{thm:split-U1xU1}
    The unreduced $U(1) \times U(1)$-equivariant Khovanov complex $\CKh_\alpha(D)$ splits into the direct sum of two copies of the reduced complex $\rCKh_\alpha(D)$ over $\Z$.  
\end{thm}
\section{Rasmussen invariant and homological generators}
\label{sec:rasmussen}

\subsection{Lee classes}
\label{subsec:lee-class}

Hereafter, we consider the $U(1)$-equivariant theory, for the Frobenius extension $(R_h, A_h)$. (An analogous argument holds for the $U(1) \times U(1)$-equivariant theory, using $\sigma_\alpha, \nu_\alpha$ defined in \Cref{subsec:U1xU1-equiv} in place of $\hsigma, \hnu$, and $X_1, X_2$ in place of $X, Y$.) Recall that the \textit{Lee cycle} $\ca(D)$ of a link diagram $D$ is a cycle in $\CKh_h(D)$ obtained from a specific coloring of the Seifert circles of $D$ by $X$ or $Y = X - h$. For instance, \Cref{fig:lee-b} depicts the Lee cycle for the diagram of \Cref{fig:lee-a} (see \cite[Definition 2.8]{Sano:2020} for a precise definition). The cycle can be interpreted as a dotted cobordism from $\emptyset$ to the Seifert resolution $D_0$ of $D$, consisting of a cup for each Seifert circle decorated with $\bdot$ or $\hdot$ (\Cref{fig:lee-c}). 
Let $\cb(D)$ denote the Lee cycle for the orientation reversed diagram $-D$ of $D$. Observe that $\cb(D)$ can be obtained from $\ca(D)$ by swapping the labels $X, Y$, so we have 
\[
    \cb(D) = \hsigma(\ca(D)). 
\]

\begin{figure}[t]
    \centering
    \begin{subfigure}[t]{0.3\linewidth}
        \centering
        \tikzset{every picture/.style={line width=0.75pt}} 

\begin{tikzpicture}[x=0.75pt,y=0.75pt,yscale=-1,xscale=1]

\clip (0,0) rectangle + (100, 100);

\draw [line width=1.5]    (70.58,41.39) .. controls (63.5,-18) and (4,25.5) .. (42.36,77.01) ;
\draw [line width=1.5]    (53.15,86.58) .. controls (94,115) and (123,32) .. (35.5,45.5) ;
\draw [line width=1.5]    (25.13,46.95) .. controls (-12.3,54.35) and (15.68,141.29) .. (69.06,54.44) ;
\draw    (54.5,13) -- (47.5,13) ;
\draw [shift={(45.5,13)}, rotate = 360] [color={rgb, 255:red, 0; green, 0; blue, 0 }  ][line width=0.75]    (10.93,-4.9) .. controls (6.95,-2.3) and (3.31,-0.67) .. (0,0) .. controls (3.31,0.67) and (6.95,2.3) .. (10.93,4.9)   ;

\end{tikzpicture}
        \caption{}
        \label{fig:lee-a}
    \end{subfigure}
    \begin{subfigure}[t]{0.3\linewidth}
        \centering
        \tikzset{every picture/.style={line width=0.75pt}} 

\begin{tikzpicture}[x=0.75pt,y=0.75pt,yscale=-1,xscale=1]

\clip (0,0) rectangle + (100, 100);

\draw  [color={rgb, 255:red, 208; green, 2; blue, 27 }  ,draw opacity=1 ][line width=1.5]  (49.5,13) .. controls (62.5,12) and (67.5,20) .. (69.5,37) .. controls (71.5,54) and (79.27,44.04) .. (89.5,55) .. controls (99.73,65.96) and (98.5,74) .. (96.5,83) .. controls (94.5,92) and (82.5,100) .. (65.5,91) .. controls (48.5,82) and (55.5,83) .. (40.5,92) .. controls (25.5,101) and (16.5,91) .. (10.5,84) .. controls (4.5,77) and (5.5,62) .. (12.5,57) .. controls (19.5,52) and (30.5,52) .. (32.5,37) .. controls (34.5,22) and (36.5,14) .. (49.5,13) -- cycle ;
\draw  [color={rgb, 255:red, 0; green, 116; blue, 255 }  ,draw opacity=1 ][line width=1.5]  (49.5,44) .. controls (61.5,44) and (75,41) .. (68.5,63) .. controls (62,85) and (45,87) .. (35,67) .. controls (25,47) and (37.5,44) .. (49.5,44) -- cycle ;

\draw (72,5.4) node [anchor=north west][inner sep=0.75pt]    {$\textcolor[rgb]{0.82,0.01,0.11}{X}$};
\draw (70.5,58.4) node [anchor=north west][inner sep=0.75pt]  [color={rgb, 255:red, 0; green, 117; blue, 255 }  ,opacity=1 ]  {$Y$};

\end{tikzpicture}
        \caption{}
        \label{fig:lee-b}
    \end{subfigure}
    \begin{subfigure}[t]{0.35\linewidth}
        \centering
        \tikzset{every picture/.style={line width=0.75pt}} 

\begin{tikzpicture}[x=0.75pt,y=0.75pt,yscale=-1,xscale=1]

\clip (0,0) rectangle + (150, 100);

\draw [color={rgb, 255:red, 208; green, 2; blue, 27 }  ,draw opacity=1 ][line width=1.5]    (98.22,13.41) .. controls (107.84,12.45) and (111.54,20.08) .. (113.02,36.28) .. controls (114.5,52.49) and (120.25,43) .. (127.82,53.44) .. controls (135.38,63.88) and (134.47,71.55) .. (132.99,80.13) .. controls (131.51,88.71) and (128.99,89.47) .. (124.14,90.98) ;
\draw [color={rgb, 255:red, 0; green, 116; blue, 255 }  ,draw opacity=1 ][line width=1.5]    (98.22,42.96) .. controls (107.1,42.96) and (117.09,40.1) .. (112.28,61.07) .. controls (107.47,82.03) and (104.49,78.54) .. (99.21,78.03) ;
\draw [color={rgb, 255:red, 208; green, 2; blue, 27 }  ,draw opacity=1 ][line width=1.5]    (98.22,13.41) .. controls (71.82,13.41) and (9,1.67) .. (9,56) .. controls (9,110.33) and (96.96,90.49) .. (124.14,90.98) ;
\draw  [fill={rgb, 255:red, 0; green, 0; blue, 0 }  ,fill opacity=1 ] (26.3,58.45) .. controls (26.3,56.76) and (27.67,55.39) .. (29.35,55.39) .. controls (31.04,55.39) and (32.41,56.76) .. (32.41,58.45) .. controls (32.41,60.13) and (31.04,61.5) .. (29.35,61.5) .. controls (27.67,61.5) and (26.3,60.13) .. (26.3,58.45) -- cycle ;
\draw [color={rgb, 255:red, 208; green, 2; blue, 27 }  ,draw opacity=1 ][line width=1.5]  [dash pattern={on 1.69pt off 2.76pt}]  (124.14,90.98) .. controls (108.4,90.99) and (108.4,81.84) .. (102.3,82.23) .. controls (96.2,82.61) and (91.24,87.94) .. (91.57,88.71) .. controls (91.89,89.47) and (75.23,88.71) .. (69.89,81.84) .. controls (64.55,74.98) and (65.67,60.11) .. (70.85,55.35) .. controls (76.03,50.58) and (84.17,50.58) .. (85.65,36.28) .. controls (87.13,21.99) and (88.28,16.13) .. (95.31,13.97) ;
\draw [color={rgb, 255:red, 0; green, 116; blue, 255 }  ,draw opacity=1 ][line width=1.5]  [dash pattern={on 1.69pt off 2.76pt}]  (100.46,78.03) .. controls (95.91,78.45) and (91.02,70.52) .. (87.43,61.26) .. controls (83.83,51.99) and (87.65,43.7) .. (95.54,43.06) ;
\draw [color={rgb, 255:red, 74; green, 144; blue, 226 }  ,draw opacity=1 ][fill={rgb, 255:red, 255; green, 255; blue, 255 }  ,fill opacity=0.8 ][line width=1.5]    (98.22,42.96) .. controls (71.82,42.96) and (49.78,48.32) .. (49.3,59.99) .. controls (48.82,71.67) and (65.84,72.77) .. (100.46,78.03) ;
\draw  [color={rgb, 255:red, 0; green, 0; blue, 0 }  ,draw opacity=1 ][fill={rgb, 255:red, 255; green, 255; blue, 255 }  ,fill opacity=1 ] (57.92,60.71) .. controls (57.92,59.02) and (59.28,57.66) .. (60.97,57.66) .. controls (62.66,57.66) and (64.02,59.02) .. (64.02,60.71) .. controls (64.02,62.4) and (62.66,63.76) .. (60.97,63.76) .. controls (59.28,63.76) and (57.92,62.4) .. (57.92,60.71) -- cycle ;

\draw (117.65,6.76) node [anchor=north west][inner sep=0.75pt]    {$\textcolor[rgb]{0.82,0.01,0.11}{X}$};
\draw (113.91,63.95) node [anchor=north west][inner sep=0.75pt]  [color={rgb, 255:red, 0; green, 117; blue, 255 }  ,opacity=1 ]  {$Y$};

\end{tikzpicture}
        \caption{}
        \label{fig:lee-c}
    \end{subfigure}
    \caption{The Lee cycle $\ca(D)$ of a diagram $D$.}
    \label{fig:lee-cycle-cob}    
\end{figure}
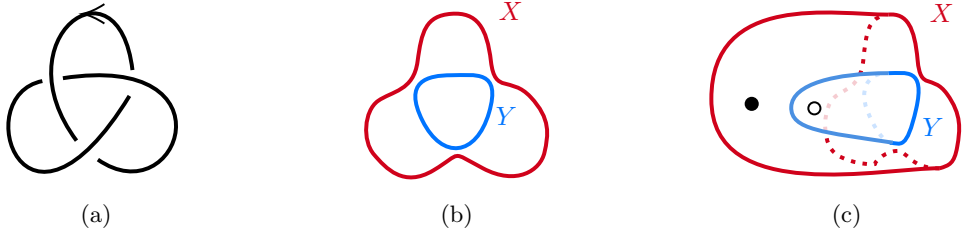

The Lee classes were originally defined by Lee in \cite{Lee:2005} for $\Q$-Lee homology; for an $\ell$-component link diagram $D$, there are $2^\ell$ Lee classes, one defined for each orientation $o$ on $D$, and the $\Q$-Lee homology of $D$ is freely generated by them. This result was extended to $\F_2$-Bar-Natan homology by Turner in \cite{Turner:2006}, and in general by Mackaay, Turner and Vaz in \cite{MTV:2007} for any link homology obtained from a rank 2 Frobenius algebra $A = R[X]/(X^2 - hX - t)$ over a commutative ring $R$ whose defining quadratic polynomial $X^2 - hX - t$ factorizes as $(X - a_1)(X - a_2)$ over $R$ and the difference of the two roots $a_2 - a_1$ is invertible in $R$ (see \cite[Proposition 2.3]{MTV:2007} or \cite[Theorem 1]{Turner:2020}). In general, the Lee classes can be defined whenever $X^2 - hX - t$ factorizes as $(X - a_1)(X - a_2)$, but does not necessarily generate the homology unless $a_2 - a_1$ is invertible. 
In particular, Plamenevskaya's invariant $\psi(L)$ of a transverse link $L$ is the Lee class in $\Kh_0(L)$ for the special case $a_1 = a_2 = 0$, which may be trivial in the homology group \cite{Plamenevskaya:2006}. 

In \cite{Sano:2020-b}, it is proved that Lee classes can be used to fix the sign indeterminacy of Khovanov homology and its equivariant versions.\footnote{
    Up to sign functoriality of Khovanov homology was first proved by Jacobsson~\cite{Jacobsson:2002} and subsequently by Bar-Natan~\cite{BarNatan:2005} in a more general framework. The sign indeterminacy was fixed in \cite{Cap:2007,CMW:2009,Blanc:2010,Beliakova:2019,Vog:2020} under various extensions of the theory. 
}
With the signs adjusted accordingly, we have the following propositions. Here, for a link diagram $D$, $w(D)$ denotes the writhe, $r(D)$ is the number of Seifert circles of $D$. For an unary function $f$, $\delta f$ denotes the difference $\delta f(x, y) := f(y) - f(x)$. 

\begin{prop}[{\cite[Proposition 3.1]{Sano:2020-b}}]
\label{prop:R-map-lee-class}
    Let $D, D'$ be link diagrams related by a Reidemeister move. Then the corresponding isomorphism between the homology groups
    \[
        \phi\colon \Kh_h(D) \to \Kh_h(D')
    \]
    maps the Lee class of $D$ to that of $D'$ multiplied by some power of $h$, 
    \[
        [\ca(D)] \mapsto h^{j} [\ca(D')]
    \]
    where the exponent $j \in \{0, \pm 1\}$ is given by the following formula
    \[
        j = \frac{\delta w(D, D') - \delta r(D, D')}{2}.
    \]
\end{prop}

\begin{prop}[{\cite[Proposition 3.4]{Sano:2020-b}}]
\label{prop:cob-map-lee-class}
    Let $S$ be a link cobordism represented as a sequence of movies between two non-empty link diagrams $D, D'$. Further suppose that every component of $S$ has boundary in $D$. The corresponding cobordism map between the homology groups (modulo torsion)
    \[
        \phi_S\colon \Kh_h(D) / \Tor \to \Kh_h(D') / \Tor
    \]
    maps the Lee class of $D$ to that of $D'$ multiplied by some power of $h$, 
    \[
        [\ca(D)] \mapsto h^{j} [\ca(D')]
    \]
    where the exponent $j \in \Z$ is given by the following formula
    \[
        j = \frac{\delta w(D, D') - \delta r(D, D') - \chi(S)}{2}.
    \]
\end{prop}

In each of the above two formulas describing the exponent $j$, whenever $j < 0$, it is understood that $[\ca(D')]$ is divisible by $h^{-j}$. 
Combining \Cref{prop:cobordism-map-and-hsigma-commute} with \Cref{prop:R-map-lee-class}, we have 
\[
    \phi\colon [\cb(D)] \mapsto (-h)^{j} [\cb(D')],
\]
recovering the second equation proved in \cite[Proposition 3.1]{Sano:2020-b}. The following commutative diagram describes the above equations:
\[
\begin{tikzcd}[column sep=4em]
R_h \arrow[d, "\hsigma"'] \arrow[r, "\ca(D)"] \arrow[rr, "h^j \ca(D')", bend left] & \Kh_h(D) \arrow[d, "\hsigma"] \arrow[r, "\phi"] & \Kh_h(D') \arrow[d, "\hsigma"] \\
R_h \arrow[r, "\cb(D)"] \arrow[rr, "(-h)^j \cb(D')"', bend right] & \Kh_h(D) \arrow[r, "\phi"] & \Kh_h(D')
\end{tikzcd}
\]
A similar equation also holds for the cobordism map $\phi_S$ of \Cref{prop:cob-map-lee-class}. 

\subsection{Rasmussen invariant}

Let $F$ be a field of any characteristic, and consider the $U(1)$-equivariant theory (Bar-Natan's theory) over $F$. The Frobenius extension is given by $R_h \otimes F = F[h]$ and $A^F_h := A_h \otimes F = F[h, X]/(X^2 - hX)$. Rasmussen's \textit{$s$-invariant} over $F$ can be defined in two ways: (i) using the unreduced Bar-Natan homology over $F$, or (ii) the reduced Bar-Natan homology over $F$  (see \cite{Rasmussen:2010,LS:2014,KWZ:2019}). Here, we reprove that the two definitions coincide by relating the homological generators of the unreduced and reduced homologies. 

For a knot $K$, its unreduced Bar-Natan homology over $F$ is known to have the form:
\[
    \Kh_h(K; F) \cong q^{-s - 1} F[h] \oplus q^{-s + 1} F[h] \oplus (\Tor).
\]
It has rank $2$ over the ring $F[h]$; the two generators are concentrated in homological grading $0$, and their quantum gradings differ by $2$. Define $-s^F(K) = -s$ to be the average of the quantum gradings of the two generators. (The negative sign is due to the convention of quantum grading.)
The reduced Bar-Natan homology over $F$ has the form:
\[
    \rKh_h(K; F) \cong q^{-\tilde{s}} F[h] \oplus (\Tor).
\]
It has rank $1$ over the ring $F[h]$; the unique generator has homological grading $0$. Define $-\tilde{s}^F(K)=-\tilde{s}$ to be the quantum grading of the generator. 

\begin{prop}
\label{prop:red-unred-corresp}
    Let $D$ be a diagram of $K$ and $z$ a cycle that represents a generator of 
    \[\rKh_h(D; F) / \Tor \cong F[h].
    \]
    Then, the two cycles $z, \hnu(z)$ give a basis of 
    \[
        \Kh_h(D; F) / \Tor \cong F[h]^2.
    \]
    In particular, this shows $s^F(K) = \tilde{s}^F(K)$. 
\end{prop}

\begin{proof}
    When $\fchar{F} = 2,$ the result is obvious from the splitting over $F[h]$. Hereafter, we assume $\fchar{F} \neq 2$. Take any point on $D$ and regard it as a pointed diagram. As in \Cref{subsec:reduced}, let $C$ denote the unreduced complex $\CKh_h(D)$ and $C_X$ and $C_Y$ the corresponding subcomplexes. Take a cycle $z = \underline{X} \otimes x$ in $C_X$ with $\gr_q(z) = \tilde{s} + 1$ that gives a generator of $H(C_X) / \Tor$. Elements $z$ and $hz$ give a basis of $H(C_X) / \Tor$ over $F[h^2]$. From \Cref{thm:split}, the four cycles
    \[
        z,\ hz,\ \hnu(z),\ \hnu(hz)
    \]
    give a basis of $H(C) / \Tor$ over $F[h^2]$. From
    \[
        \hnu(hz) = 2z - h \hnu(z),
    \]
    we can instead choose 
    \[
        z,\ hz,\ \hnu(z),\ h\hnu(z)
    \]
    as a basis of $H(C) / \Tor$. Therefore, over $F[h]$, the two cycles
    \[
        z,\ \hnu(z)
    \]
    give a basis of $H(C) / \Tor$. In particular,
    \[
        \gr_q(z) = -\tilde{s} + 1,\ \gr_q(\hnu(z)) = -\tilde{s} - 1
    \]
    shows that $s^F = \tilde{s}^F$. 
\end{proof}

\begin{remark}
    It is known that $s^F$ depends on the field $F$; in fact, direct computation shows that $s^{\Q}, s^{\F_2}, s^{\F_3}$ are linearly independent (see \cite[Remark 6.1]{LS:2014}, \cite[Section 6]{Schuetz:2022} and \cite{LZ:2021}). Whether the infinite set $\{s_F\}$ of the Rasmussen invariants is linear independent as $F$ runs over all prime fields remains open (\cite[Question 6.1]{LS:2014}). For an arbitrary field $F$, it is proved in \cite[Proposition 4.36]{Sano-Sato:2023} that $s^F$ depends only on the characteristic of $F$. 
\end{remark}

\begin{ex}
    Consider the simplest case $D = \bigcirc$. The reduced homology $H(C_X) = C_X \cong F[h]$ is generated by $\underline{X}$. \Cref{prop:red-unred-corresp} implies that the unreduced homology $H(C) = C \cong A^F_h$ is generated by $X$ and $\hnu(X) = 1$, which is obviously true. 
\end{ex}

\Cref{prop:red-unred-corresp} can be easily generalized to links. 

\begin{prop}
\label{prop:red-unred-corresp-link} 
    Let $D$ be an $\ell$-component link diagram whose reduced homology $\rKh_h(D; F)$ has at most rank 1 in each homological grading. Let $z_1, \ldots, z_{2^{\ell - 1}}$ be $2^{\ell - 1}$ cycles that give a basis of $\rKh_h(D; F) / \Tor \cong F[h]^{2^{\ell - 1}}$. Then the $2^{\ell}$ cycles $z_1, \hnu(z_1), \ldots, z_{2^{\ell - 1}}, \hnu(z_{2^{\ell - 1}})$ give a basis of 
    \[\Kh_h(D; F) / \Tor \cong F[h]^{2^\ell}.
    \]
\end{prop}

\begin{ex}
    \begin{figure}
        \centering
        \input{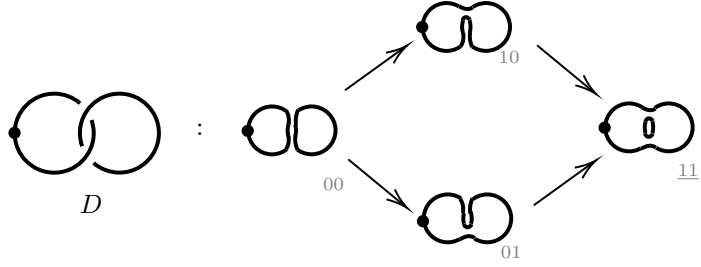}
        \caption{A Hopf link diagram $D$ and its cube of resolutions.}
        \label{fig:hopf-link}
    \end{figure}
    Let $D$ be a positive Hopf link diagram. Choose a basepoint on one of its components, and regard it as a pointed diagram, as in \Cref{fig:hopf-link}. A basis of the reduced homology $\rKh_h(D; F) \cong F[h]^2$ is given by cycles
    \[
        z_1 = \underline{X} \otimes Y,\ 
        z_2 = \underline{X} \otimes 1
    \]
    in homological grading $0$ and $2$ respectively. Here, the underline indicates the label on the pointed circle. From \Cref{prop:red-unred-corresp-link}, a basis of the unreduced homology $\Kh_h(D; F) \cong F[h]^4$ is given by the four cycles
    \begin{gather*}
        z_1 = \underline{X} \otimes Y,\quad \hnu(z_1) = \underline{1} \otimes X - \underline{X} \otimes 1, \\
        z_2 = \underline{X} \otimes 1,\quad \hnu(z_2) = \underline{1} \otimes 1.
    \end{gather*}
\end{ex}

\subsection{Describing the homological generators}

Next, we show that for a knot diagram $D$, the generators of $\Kh_h(D; F)/\Tor \cong F[h]^2$ and $\rKh_h(D; F)/\Tor \cong F[h]$ can be described using the Lee classes. First, choose a base point on $D$, and choose either one of the subcomplexes $C_X$, $C_Y$ of $\CKh_h(D)$ that contains the Lee cycle $\ca(D)$ to be the reduced complex $\rCKh_h(D)$. Take a cycle $z$ that represents a generator of $\rKh_h(D; F)/\Tor \cong F[h]$, such that the Lee class (modulo torsion) can be written as 
\[
    [\ca(D)] = h^d [z]
\]
for some integer $d \geq 0$. This integer $d$ is called the \textit{$h$-divisibility} of the Lee class $[\ca(D)]$ over $F$, and is denoted $d_h(D)$. Passing to the unreduced homology, we have 
\[
    \hnu[\ca(D)] = \begin{cases}
        h^d [\hnu(z)] & \text{\ if $d$ is even,} \\ 
        h^{d - 1} (2[z] - h [\hnu(z)]) & \text{\ if $d$ is odd}.
    \end{cases} 
\]
On the other hand, the definition of $\hnu$ gives
\[
    \hnu(\ca(D)) = \frac{\ca(D) - \cb(D)}{h}.
\]
Thus we have 
\[
    [\hnu(z)] = \frac{[\ca(D)] + (-1)^{d + 1} [\cb(D)]}{h^{d + 1}}.
\]
\Cref{prop:red-unred-corresp} states that $[z]$ and $[\hnu(z)]$ generates $\Kh_(D; F) / \Tor$. Moreover, the following proposition states that these classes are in fact knot invariants. 

\begin{prop}
\label{prop:cz-invariance}
    The following classes in $\Kh_h(D; F)/\Tor$,
    \[
        [\ca(D)]/h^d,\quad 
        [\cb(D)]/(-h)^d
    \]
    and
    \[
        \frac{[\ca(D)] + (-1)^{d + 1} [\cb(D)]}{h^{d + 1}}
    \]
    are invariant under the Reidemeister moves.
\end{prop}

\begin{proof}
    Let $D'$ be another diagram representing the same knot, and $\phi$ be the isomorphism between the two homology groups. From \Cref{prop:R-map-lee-class}, we have 
    \[
        \phi([\ca(D)]) = h^j [\ca(D')],\quad 
        \phi([\cb(D)]) = (-h)^j [\cb(D')]
    \]
    where 
    \[
        j = \frac{\delta w(D, D') - \delta r(D, D')}{2}.
    \]
    From \cite[Theorem 1]{Sano:YUI}, the quantity
    \[
        2d_h(D) + w(D) - r(D)
    \]
    is invariant under the Reidemeister moves. This shows that
    \[
        \phi([\ca(D)]/h^d) =  [\ca(D')]/h^{d'},\quad 
        \phi([\cb(D)]/(-h)^{d'}) = [\cb(D')]/(-h)^{d'}
    \]
    where $d' = d_h(\ca(D'))$ denotes the $h$-divisibility of the Lee class of $D'$.
    The latter statement is immediate from the former. 
\end{proof}

\Cref{prop:cz-invariance} justifies to write 
\[
    \widetilde{\cz}(K) := [\ca(D)]/h^d,\quad
    \cz(K) := \frac{[\ca(D)] + (-1)^{d + 1} [\cb(D)]}{h^{d + 1}}
\]
for a knot $K$ with diagram $D$. In summary, 

\begin{prop}
\label{prop:kh-free-generators}
    Let $F$ be a field and $K$ a knot. Let $s = s^F(K)$ be the Rasmussen invariant of $K$ over F. 
    \begin{enumerate}
        \item $\rKh_h(K)/\Tor \cong q^{-s} F[h]$ is freely generated by $\widetilde{\cz}(K)$. 

        \item $\Kh_h(K)/\Tor \cong q^{-s - 1} F[h] \oplus q^{-s + 1} F[h]$ is freely generated by $\cz(K)$ and $\widetilde{\cz}(K)$. 
    \end{enumerate}
\end{prop}

\begin{remark}
    As in \Cref{prop:red-unred-corresp-link}, a similar description for a link $L$ can be given using the Lee classes, provided that reduced homology $\rKh_h(L; F)$ has at most rank 1 in each homological grading.
\end{remark}

It is easy to see that $\gr_q(\ca(D)) = -w(D) + r(D)$, and with $\deg(h) = 2$, we have 
\begin{align*}
    s^F(K) 
    &= -\gr_q(\cz(D)) - 1 \\
    &= 2d_h(D) + w(D) - r(D) + 1,
\end{align*}
recovering the formulas of \cite[Theorem 3]{Sano:2020} for the unreduced case, and \cite[Theorem 2]{Sano-Sato:2023} for the reduced case. 

We show that the classes $\cz(K), \widetilde{\cz}(K)$ behave well with respect to cobordisms. 
Consider the element $U := X + Y$. From $X^2 = hX$, $Y^2 = -hY$ and $XY = 0$, we have 
\[
    UX = hX, \quad UY = -hY.
\]
Let $D$ be a knot diagram with base point $p$, and $C$ the Seifert circle of $D$ that contains $p$. Recall that $C$ is either labeled $X$ or $Y$ with respect to the $XY$-labeling that defines the Lee cycle. Define an endomorphism $u$ on $\CKh(D)$ by 
\[
    u (\underline{x_1} \otimes x_2 \otimes \cdots) := \begin{cases}
        \ \ \underline{Ux_1} \otimes x_2 \otimes \cdots & \text{ if $C$ is labeled $X$,} \\ 
        -\underline{Ux_1} \otimes x_2 \otimes \cdots & \text{ if $C$ is labeled $Y$}.
    \end{cases}
\]
Then we have 
\[
    u \ca(D) = h \ca(D), \quad u \cb(D) = -h \cb(D).
\]
It follows from \cite[Proposition 3.2]{Sano-Sato:2023} that $u$ is independent of the choice of the base point, up to chain homotopy. From $u^2 = h^2$ it follows that $\Kh_h(D)$ admits a $F[u] / (u^2 - h^2)$ module structure. With \Cref{prop:cob-map-lee-class}, we immediately obtain the following propositions. 

\begin{prop}
\label{prop:cz-cobordism}
    Let $S$ be an oriented, connected cobordism between knots $K, K'$. The corresponding cobordism map 
    \[
        \phi_S\colon \Kh_h(K; F) / \Tor \to \Kh_h(K'; F) / \Tor
    \]
    sends
    \[
        \cz(K) \mapsto u^j \cz(K'),\quad 
        \widetilde{\cz}(K) \mapsto u^j \widetilde{\cz}(K'),\quad 
    \]
    where $j \geq 0$ is given by 
    \[
        j = \frac{\delta s^F(K, K') - \chi(S)}{2}.
    \]
\end{prop}

\begin{cor}
    $\cz(K), \widetilde{\cz}(K)$ are knot concordance invariants. 
\end{cor}

\subsection{Characteristic \texorpdfstring{$\neq 2$}{neq 2} and \texorpdfstring{$SU(2)$}{SU(2)}-equivariant theory}

Assume throughout this section that $\fchar{F} \neq 2$. In this case, we can alternatively take $\cz(D)$ and 
\[
    \cz'(D) := u \cz(D) = \frac{[\ca(D)] + (-1)^d [\cb(D)]}{h^d}
\]
as generators of $\Kh_h(D; F)/\Tor$ over $F[h]$. With $u^2 = h^2$, so we can equivalently state that $\Kh_h(D; F)/\Tor$ is freely generated by $\cz(D)$ and $h \cz(D)$ over $F[u]$. This decomposition shows the $\hsigma$-symmetry of $\Kh_h(D; F)/\Tor$ more clearly: $F[u] \langle \cz(D) \rangle$ is the $(+1)$-eigenspace, and $F[u] \langle h \cz(D) \rangle$ is the $(-1)$-eigenspace of $\hsigma$. The following diagram depicts how the two decompositions of $\Kh_h(D; F)/\Tor$ are related.
\[
    \begin{tikzcd}[row sep=2.5em]
\ \vdots\ & \ \vdots\ \\
h^2 \cz(D) \arrow[u, "h", bend left] \arrow[ru] & h \cz'(D) \arrow[u, "h"', bend right] \arrow[lu, "u" description, dashed] \\
h \cz(D) \arrow[u, "h", bend left] \arrow[ru, dashed] & \cz'(D) \arrow[u, "h"', bend right] \arrow[lu, "u" description] \\
\cz(D) \arrow[u, "h", bend left] \arrow[ru, "u" description] &                           
    \end{tikzcd}
\]

Next we prove an analogous result for the $SU(2)$ equivariant theory (bigraded Lee theory) over $F$. A similar argument is given in \cite[Section 3.3]{Sano:2020}, but we restate it here for completeness. The Frobenius extension is given by $R_t \otimes F = F[t]$ and $A^F_t = F[t, X]/(X^2 - t)$. Also consider the rank 2 extension $F[\sqrt{t}]$ of $F[t]$ and $A^F_{\sqrt{t}} := A^F_t \otimes F[\sqrt{t}]$. For a knot diagram $D$, let $\CKh_t(D; F)$ and $\CKh_{\sqrt{t}}(D; F)$ denote the corresponding chain complexes, regarded over $F[t]$ and $F[\sqrt{t}]$ respectively. 
We have 
\[
    \CKh_{{\sqrt{t}}}(D; F) = \CKh_t(D; F) \oplus \sqrt{t} \CKh_t(D; F) 
\]
over $F[t]$. Moreover, since $\deg t = 4$, $\CKh_t(D; F)$ splits as 
\[
    \CKh_t(D; F) = \CKh^{[1]}_t(D; F) \oplus \CKh^{[-1]}_t(D; F)
\]
where $\CKh^{[\pm 1]}_t(D; F)$ denotes the quantum grading $\pm 1 \bmod{4}$ subcomplex of $\CKh_t(D; F)$. 

Let $\ca_{\sqrt{t}}(D)$, $\cb_{\sqrt{t}}(D)$ denote the two Lee cycles in $\CKh_{\sqrt{t}}(D; F)$, given by tensor products of elements
\[
    X_{\pm} := X \pm \sqrt{t} \in A^F_{\sqrt{t}}.
\]
Although these cycles do not belong to $\CKh_t(D; F)$, we have the following. 

\begin{lem}
    $\ca_{\sqrt{t}}(D) - \cb_{\sqrt{t}}(D)$ is divisible by $\sqrt{t}$, and the two elements
    \[
        \cx_t(D) := \ca_{\sqrt{t}}(D) + \cb_{\sqrt{t}}(D),\quad 
        \cy_t(D) := \frac{\ca_{\sqrt{t}}(D) - \cb_{\sqrt{t}}(D)}{2\sqrt{t}}
    \]
    belong to $\CKh_t(D; F)$. Moreover, one is contained in $\CKh^{[1]}_t(D; F)$ and the other in $\CKh^{[-1]}_t(D; F)$.
\end{lem}

\begin{proof}
    Take an element $x = x_1 \otimes \cdots \otimes x_r \in (A^F_{\sqrt{t}})^{\otimes r}$ with $x_i \in \{X_\pm\}$. It suffices to prove 
    \[
        x + \hsigma(x) \in (A^F_t)^{\otimes r},\quad 
        x - \hsigma(x) \in \sqrt{t} (A^F_t)^{\otimes r}.
    \]
    For $r = 1$, we have 
    \[
        X_+ + X_- = 2X,\quad 
        X_+ - X_- = 2\sqrt{t}.
    \]
    For $r > 1$, put $x' = x_2 \otimes \cdots \otimes x_r$. Assuming $x_1 = X + \sqrt{t}$, we have 
    \begin{align*}
        x \pm \hsigma(x) 
        &= (X + \sqrt{t}) \otimes x' \pm (X - \sqrt{t}) \otimes \hsigma(x') \\
        &= X \otimes (x' \pm \hsigma(x')) + \sqrt{t} \otimes (x' \mp \hsigma(x'))
    \end{align*}
    and the proof follows by induction. 
\end{proof}

Identify rings $F[h]$ and $F[\sqrt{t}]$ by the correspondence
\[
    F[h] \to F[\sqrt{t}],\quad h \mapsto 2\sqrt{t}
\]
and consider the ring isomorphism
\[
    F[h, X] \to F[\sqrt{t}, X];\quad 
    X \mapsto X + \sqrt{t}.
\]
This induces an involutive Frobenius algebra isomorphism
\[
    A^F_h = F[h, X]/(X^2 - hX) \to A^F_{\sqrt{t}} = F[\sqrt{t}, X]/(X^2 - t)
\]
and a chain isomorphism
\[
    \CKh_h(D; F) \to \CKh_{\sqrt{t}}(D; F).
\]
Consider the two homology classes
\[
    \cz_t(D) := \frac{[\ca_{\sqrt{t}}(D)] + (-1)^{d + 1} [\cb_{\sqrt{t}}(D)]}{(2\sqrt{t})^{d + 1}},\quad 
    \cz'_t(D) := \frac{[\ca_{\sqrt{t}}(D)] + (-1)^d [\cb_{\sqrt{t}}(D)]}{(2{\sqrt{t}})^d}
\]
in $\Kh_{\sqrt{t}}(D; F) / \Tor$. 

\begin{lem}
    The two homology classes $\cz_t(D),\ \cz'_t(D)$ belong to $\Kh_t(D; F) / \Tor$, where we regard 
    \[
        \Kh_t(D; F) \subset \Kh_{\sqrt{t}}(D; F) = \Kh_t(D; F) \oplus \sqrt{t} \Kh_t(D; F)
    \]
    over $F[t]$. Moreover, one is contained in $\Kh^{[1]}_t(D; F)$ and the other in $\Kh^{[-1]}_t(D; F)$.
\end{lem}

\begin{proof}
    If $d$ is even,
    \[
        \cz_t(D) = \frac{[\cy_t(D)]}{(4t)^{d/2}},\quad 
        \cz'_t(D) = \frac{[\cx_t(D)]}{(4t)^{d/2}}
    \]
    and if $d$ is odd, 
    \[
        \cz_t(D) = \frac{[\cx_t(D)]}{(4t)^{(d + 1)/2}},\quad 
        \cz'_t(D) = \frac{[\cy_t(D)]}{(4t)^{(d - 1)/2}}.
    \]
\end{proof}

It is immediate from \Cref{prop:cz-invariance} that both $\cz_t(D), \cz'_t(D)$ are invariant under the Reidemeister moves. Thus, for a knot $K$, $\Kh_t(K; F) / \Tor \cong F[t]^2$ is generated by the two classes $\cz_t(K), \cz'_t(K)$ over $F[t]$. The endomorphism $u$ on $\CKh_t(K; F)$ is given by 
\[
    u = \pm (\Xbar_+ + \Xbar_-) = \pm 2 \Xbar
\]
and with $u^2 = 4t$, $\Kh_t(K; F) / \Tor$ can be regarded as a free $F[u]$ module generated by $\cz_t(K)$. The following diagram depicts how the two generators of $F[t]^2$ and the single generator of $F[u]$ are related.
\[
\begin{tikzcd}
\ \ \vdots\ \ & \\ 
& 4t \cz'_t(K) \arrow[lu, "u" description] \\
4t \cz_t(K) \arrow[ru, "u" description] \arrow[uu, "4t", bend left] & \\ 
& \cz'_t(K) \arrow[lu, "u" description] \arrow[uu, "4t"', bend right] \\
\cz_t(K) \arrow[ru, "u" description] \arrow[uu, "4t", bend left] &
\end{tikzcd}
\]
In summary, we recover \cite[Corollary 3.41]{Sano:2020}, generalizing \cite[Proposition 8]{Khovanov:2004}.

\begin{prop}
    If $\fchar{F} \neq 2$, $\Kh_t(K; F) / \Tor \cong q^{-s - 1} F[u]$ is freely generated by $\cz_t(K)$.
\end{prop}

\begin{remark}
    In \cite{QRSW:2023}, Qi, Robert, Sussan and Wagner construct an $\sl_2$-action on the equivariant $\mathfrak{gl}_N$ Khovanov--Rozansky homology, and, in particular, characterize the Rasmussen invariant as the highest weight of a certain quotient representation \cite[Section 6.3]{QRSW:2023}. It would be interesting to describe the $\sl_2$-action on equivariant Khovanov homology, and relate it with the descriptions obtained above. One can ask whether the $h$-divisibility $d_h(D)$ is related to the maximum $d$ such that $\mathsf{f}^d [\ca(D)] \neq 0$. 
\end{remark}

\begin{remark}
    The torsion part of equivariant Khovanov homology has also significant topological applications, as shown in \cite{Alishahi:2017,Alishahi:2018,Sarkar:2020,Onkar:2020,Caprau-etal:2021,Zhuang:2022,Hayden:2023,Lewark-Marino-Zibrowius:2024,Iltgen-Lewark-Marino:2025}. It is interesting to ask whether the generators of the torsion part can also be given explicit descriptions, and whether there is a canonical splitting of the homology group into the torsion part and the free part. 
\end{remark}

\medskip
\noindent 
\textbf{Acknowledgements.}
M.K. was partially supported by NSF grant DMS-2204033 and Simons
Collaboration Award 994328 ``New Structures in Low-Dimensional Topology''.
T.S. was supported by JSPS KAKENHI Grant Number 23K12982 and academist crowdfunding Project No.\ 121.

\printbibliography

\end{document}